\documentclass[a4papet,10pt]{article}
\usepackage[italian, english]{babel}
\usepackage{amsfonts,amssymb,mathtools}
\usepackage{graphicx}
\usepackage{amsmath}
\usepackage[latin1]{inputenc} 
\usepackage{bbm}
\usepackage{pstricks}
\usepackage{psfrag}
\usepackage{appendix}
\usepackage{amsthm}
\usepackage{dsfont}
\usepackage{stackrel}
\usepackage{eufrak}
\usepackage{mathrsfs}
\usepackage{enumerate}
\usepackage{cite}
\usepackage{chngcntr}
\counterwithin{figure}{section}
\usepackage[a4paper,top=30mm,bottom=30mm,left=30mm,right=30mm]{geometry}
\usepackage{appendix}

\numberwithin{equation}{section}

\newcommand{\prob}{\mathbb{P}}
\newcommand{\Ex}{\mathbb{E}}
\newcommand{\Rl}{\mathbb{R}}
\newcommand{\Cm}{\mathbb{C}}

\newcommand{\Qseq}{\{\mathsf{Q}_N\}_{N\ge 1}}
\newcommand{\X}{\{X_n\}_{n\ge 1}}

\DeclareMathOperator*{\infp}{\inf\vphantom{p}}

\newtheoremstyle{mystyle}
  {}
  {}
  {\itshape}
  {}
  {\bfseries}
  {.}
  { }
  {\thmname{#1}\thmnumber{ #2}\thmnote{ (#3)}}

\theoremstyle{mystyle}
\newtheorem{theorem}{Theorem}[section]

\newtheorem{lemma}{Lemma}[section]
\newtheorem{proposition}{Proposition}[section]

\newtheorem{definition}{Definition}[section]
\newtheorem{example}{Example}[section]
\newtheorem{remark}{Remark}[section]

\title{Large deviations for quadratic functionals of stable Gauss-Markov chains and entropy production}

\author{Marco Zamparo\footnote{Dipartimento di Fisica, Universit\`a degli Studi di Bari and INFN, Sezione di Bari,
    via Amendola 173, 70126 \newline \phantom{aaz} Bari, Italy. E-mail: \texttt{marco.zamparo@uniba.it}}

  \and
  
Massimiliano Semeraro\footnote{Dipartimento di Fisica, Universit\`a degli Studi di Bari and INFN, Sezione di Bari,
    via Amendola 173, 70126  \newline \phantom{aaz} Bari, Italy. E-mail: \texttt{massimiliano.semeraro@uniba.it}}}

\date{}

\begin{document}  
\maketitle

\begin{abstract}
In this paper we establish a large deviation principle for the entropy
production rate of possible non-stationary, centered stable
Gauss-Markov chains, verifying the Gallavotti-Cohen symmetry. We reach
this goal by developing a large deviation theory for quasi-Toeplitz
quadratic functionals of multivariate centered stable Gauss-Markov
chains, which differ from a perfect Toeplitz form by the addition of
quadratic boundary terms. 
\end{abstract}

\section{Introduction}

Let on a probability space $(\Omega,\mathscr{F},\prob)$ be given a
sequence $\X$ of random variables taking values in a Polish space
$\mathcal{X}$. For any integer $N\ge 1$, let
$\mu_N^+:=\prob[(X_1,\ldots,X_N)\in\cdot\,]$ and
$\mu_N^-:=\prob[(X_N,\ldots,X_1)\in\cdot\,]$ be the probability
measures on the Borel $\sigma$-field $\mathscr{B}(\mathcal{X}^N)$
induced by the direct process and the reverse process, respectively.
The \textit{entropy production rate} up to time $N$ is the real random
variable on $(\Omega,\mathscr{F},\prob)$ defined by
\begin{equation*}
e_N:=\begin{cases}
\frac{1}{N}\ln\Big[\frac{d\mu_N^+}{d\mu_N^-}(X_1,\ldots,X_N)\Big] & \mbox{if }
\mu_N^+\ll \mu_N^-,\\
+\infty & \mbox{otherwise}.
\end{cases}
\end{equation*}
The entropy production rate turns out to be a natural measure of
irreversibility since $e_N=0$ for all $N\ge 1$ if and only if the
sequence $\X$ is reversible, namely if and only if $(X_N,\ldots,X_1)$
is distributed as $(X_1,\ldots,X_N)$ for every $N$.  The use of the
entropy production rate to quantify the irreversibility of a
stochastic process was proposed by Kurchan \cite{Kurchan1} and in more
generality by Lebowitz, Spohn, and Maes \cite{LebowitzSpohn1,Maes1},
who extended the seminal work by Gallavotti and Cohen
\cite{GallavottiCohen1} in the context of deterministic dynamical
systems. Since then, the entropy production rate has become a basic
topic in non-equilibrium statistical physics
\cite{Searles1,Maes3,Maes2,BookEPR,Reid1,Ge1,Tome1,Landi1,Wang2015}.
The entropy production rate came out with a supposed symmetry
associated with its large fluctuations, which in fact was discovered
by Gallavotti and Cohen \cite{GallavottiCohen1} prompted by results of
computer simulations \cite{Evans1}. They dubbed this symmetry a
``fluctuation theorem''.  The appropriate formalism for describing the
large fluctuations of the entropy production rate is large deviation
theory \cite{Dembo,Frank1}.  The entropy production rate $e_N$ is said
to satisfy a \textit{large deviation principle} with the \textit{rate
  function} $I$ if there exists a function $I$ with compact level sets
such that for each Borel set $\mathcal{B}\subseteq\Rl$
\begin{equation*}
  -\inf_{w\in \mathcal{B}^o}\{I(w)\}\le\liminf_{N\uparrow\infty}\frac{1}{N}\ln\prob\big[e_N\in \mathcal{B}\big]
  \le \limsup_{N\uparrow\infty}\frac{1}{N}\ln\prob\big[e_N\in \mathcal{B}\big]\le -\inf_{w\in \bar{\mathcal{B}}}\{I(w)\},
\end{equation*}
where $\mathcal{B}^o$ and $\bar{\mathcal{B}}$ are the interior and the
closure of $\mathcal{B}$, respectively. The ``fluctuation theorem''
refers to a property of the function $I$. The rate function $I$ is
said to satisfy the \textit{Gallavotti-Cohen symmetry} if for all
$w\in\Rl$
\begin{equation*}
I(-w)=I(w)+w.
\end{equation*}
It has been pointed out that the Gallavotti-Cohen symmetry is an
intrinsic property of $I$, which is met whenever $e_N$ satisfies a
large deviation principle \cite{Ge1,BookEPR,Reid1}.

In this paper we investigate the large fluctuations of the entropy
production rate and the Gallavotti-Cohen symmetry for a possible
non-stationary, centered stable Gauss-Markov chain $\X$ valued in
$\mathcal{X}:=\Rl^d$ with any dimension $d\ge 1$. Thus, we assume that
there exists a \textit{drift matrix} $S\in\Rl^{d\times d}$ with
spectral radius $\rho(S)<1$ such that
\begin{equation}
  X_{n+1}=SX_n+G_n
  \label{evol_eq}
\end{equation}
for all $n\ge 1$, $\{G_n\}_{n\ge 1}$ being a sequence of
i.i.d.\ standard Gaussian random vectors valued in $\Rl^d$ and
independent of $X_1$. We suppose that $X_1$ is a Gaussian random
vector with mean zero and general positive-definite covariance matrix
$\Sigma_o$. The process $\X$ is stationary if and only if
$\Sigma_o=\Sigma_s:=\sum_{k\ge 0}S^k(S^\top)^k$, and it is reversible
if and only if $S$ is symmetric and $\Sigma_o=\Sigma_s$ \cite{Osawa}.
Stability corresponds to the hypothesis $\rho(S)<1$, which implies
that $\Sigma_s$ actually exists. Unstable Gauss-Markov chains for
which $\rho(S)\ge 1$ cannot achieve a stationary regime and obey a
substantially different mathematics, for which reason we leave them as
a future project.


The large deviation principle for the entropy production rate and the
Gallavotti-Cohen symmetry have been rigorously established for finite
Markov chains with discrete and continuous time \cite{BookEPR,Jiang1}.
The same has been done for multivariate stationary Ornstein-Uhlenbeck
processes with normal drift matrix
\cite{Jaksic1,Jaksic2,Budhiraja2021}, i.e.\ with a drift matrix that
commutes with its adjoint, and for a model of heat conduction through
a chain of anharmonic oscillators coupled to two reservoirs at
different temperatures \cite{Rey-Bellet1}. The mathematical tool
underlying these works is the G\"artner-Ellis theorem
\cite{Frank1,Dembo}, and the lack of a large deviation principle for
the entropy production rate of more general stochastic processes is
due to non-satisfiability of the hypotheses of that theorem.  An
attempt to overcome the limitations of the G\"artner-Ellis theorem has
been done for stationary diffusion processes, for which a large
deviation principle for the entropy production rate has been obtained
in the limit of vanishing noise by resorting to the classical
Freidlin-Wentzell theory \cite{Bertini1}.  The autoregressive model
(\ref{evol_eq}) we consider basically is the discrete-time version of
a $d$-dimensional centered Ornstein-Uhlenbeck process. The main
contribution of our work stems from the fact that we do not assume
that the drift matrix $S$ is normal or that the chain is
stationary. This generality prevents the use of the G\"artner-Ellis
theorem to get at a large deviation principle for the entropy
production rate $e_N$. The way we go around this key point is to
regard $e_N$ as a quadratic functional, and to establish a large
deviation principle for the class of quadratic functionals to which
the entropy production rate belongs via a time-dependent change of
probability measure. We need such a general principle to also tackle a
problem of large fluctuations in an active matter model \cite{noi}.
The following lemma provides the explicit expression of $e_N$ for the
model (\ref{evol_eq}) as a quadratic form. The simple proof is
reported in Appendix \ref{proof:quad}. We denote by
$\langle\cdot,\cdot\rangle$ the standard inner product of $\Rl^d$.
\begin{lemma}
  \label{lem:quad}
  Let $\X$ be a $d$-dimensional centered Gauss-Markov chain with drift
  matrix $S$ and initial positive-definite covariance $\Sigma_o$.  For
  each $N\ge 1$
  \begin{align}
    \nonumber
    Ne_N&=\frac{1}{2}\big\langle X_1,(I-\Sigma_o^{-1}-S^\top S)X_1\big\rangle+\frac{1}{2}\big\langle X_N,(\Sigma_o^{-1}+S^\top S-I)X_N\big\rangle \\
    \nonumber
    &+\sum_{n=2}^N\big\langle X_n,(S-S^\top)X_{n-1}\big\rangle.
\end{align}
\end{lemma}

Lemma \ref{lem:quad} shows that the entropy production $Ne_N$ of the
process $\X$ is a particular instance of a quasi-Toeplitz quadratic
functional $W_N$ having the form
\begin{equation}
\label{def:W}
W_N:=\frac{1}{2}\langle X_1,LX_1\rangle+\frac{1}{2}\sum_{n=2}^{N-1}\langle X_n,UX_n\rangle+\frac{1}{2}\langle X_N,RX_N\rangle+\sum_{n=2}^N\langle X_n,VX_{n-1}\rangle,
\end{equation}
$L$, $U$, $R$, and $V$ being four matrices in $\Rl^{d\times d}$ with
$L$, $U$, and $R$ symmetric. In fact, $W_N$ turns out to be $Ne_N$
when $L:=I-\Sigma_o^{-1}-S^\top S$, $U:=0$, $R:=\Sigma_o^{-1}+S^\top
S-I$, and $V:=S-S^\top$.  The circumstance to be stressed is that the
problem of the entropy production rate leads to perturb a perfect
Toeplitz structure by quadratic boundary terms, in such a way that the
coefficient matrix of $W_N$ differs from a block tridiagonal Toeplitz
matrix by the first and last diagonal blocks. This circumstance
required new large deviation principles for quadratic forms of
Gauss-Markov chains to be developed, and to be added to the literature
on large deviations for Gaussian processes.  Similarly to $e_N$, we
say that $W_N/N$ satisfies a \textit{large deviation principle} with
the \textit{rate function} $I$ if there exists a function $I$ with
compact level sets such that for each Borel set
$\mathcal{B}\subseteq\Rl$
\begin{equation*}
  -\inf_{w\in \mathcal{B}^o}\{I(w)\}\le\liminf_{N\uparrow\infty}\frac{1}{N}\ln\prob\bigg[\frac{W_N}{N}\in \mathcal{B}\bigg]
  \le \limsup_{N\uparrow\infty}\frac{1}{N}\ln\prob\bigg[\frac{W_N}{N}\in \mathcal{B}\bigg]\le -\inf_{w\in \bar{\mathcal{B}}}\{I(w)\}.
\end{equation*}
Large deviation principles for Gaussian processes have been an active
field of research since the pioneering works by Donsker and Varadhan
\cite{Donsker1985} and Bryc and Dembo \cite{Bryc1995} on the large
fluctuations of empirical measures for stationary Gaussian processes.
The focus soon moved to large deviations of quadratic functionals
\cite{Benitz1990,Dembo_Sol}, which in general cannot be tackled by a
direct application of the G\"artner-Ellis theorem since steepness of
the asymptotic cumulant generating function is not guaranteed. The
asymptotic cumulant generating function does not contain apparently
the whole information on the large deviation property of the process:
there is a loss of information passing to the limit.  For Toeplitz
quadratic forms of stationary centered Gaussian sequences, large
deviation principles are now well-established
\cite{Bercu1997,Gamboa1999,Ihara2000}, as well as some moderate
deviation principle \cite{Kakizawa2007}. These results have been
obtained by a sharp study of the spectrum of a product of two Toeplitz
matrices. For stationary centered Gaussian sequences, large deviations
have been also characterized for special Hermitian quadratic forms
\cite{Bercu1997,Bercu2000} and a sample path large deviation principle
has been deduced for the squares of the process \cite{Zani2013}.
Despite this progress, there are no general results to deal with
non-stationarity Gaussian sequences and perturbations of Toeplitz
quadratic functionals, which pose very specific problems.

The typical value of $W_N/N$ in the large $N$ limit is described by
the law of large numbers
\begin{equation*}
\lim_{N\uparrow\infty}\frac{W_N}{N}=\frac{1}{2}\mbox{tr}\Big[\big(U+V^\top S+S^\top V\big)\Sigma_s\Big]~~~~~\prob\mbox{-a.s.},
\end{equation*}
which easily follows by specializing to our model standard results of
the theory of Markov chains \cite{Meyn1}. It does not depend on the
initial condition and on the boundary terms.  On the contrary, we
shall see that non-stationarity and quadratic boundary terms, which
are intimately related for Gauss-Markov chains, affect deviations of
quadratic functionals from the mean and shape rate functions. This is
not surprising since squares of Gaussian random variables have an
exponential, rather than super-exponential, tail probability.
Researchers have already come across this issue.  In fact, the
maximum likelihood estimator and the Yule-Walker estimator for the
drift parameter of a one-dimensional autoregressive stable process
satisfy large deviation principles with different rate functions
\cite{Bercu1997}. These two estimators are connected to quadratic
functionals that differ exactly by a quadratic boundary term.  A
similar phenomenology holds for the continuous-time counterpart,
i.e.\ the Ornstein-Uhlenbeck process \cite{FlorensLandais1999}. Coming
more close to statistical physics, perturbations by quadratic boundary
terms of the entropy production rate for Ornstein-Uhlenbeck processes
with normal drift matrix have been considered to account for the heat
dissipation rate of a network of thermally driven harmonic oscillators
\cite{Jaksic2}.

The remainder of the paper is organized as follows. In Section
\ref{sec:mainres} we present the main results of this work: a large
deviation principle for quadratic functionals of type (\ref{def:W}) in
the context of multivariate centered stable Gauss-Markov chains and
the Gallavotti-Cohen symmetry of the entropy production rate
function. In Section \ref{sec:normal} we apply the theory to the very
special class of stable Gauss-Markov chains with normal drift matrix,
making contact with previous results. Section \ref{sec:proofLDP}
provides the proof of the large deviation principle for the quadratic
functionals. Section \ref{sec:simmetria} reports the proof of the
Gallavotti-Cohen symmetry of the entropy production rate function.

\subsection{Main results}
\label{sec:mainres}

From now on we regard $\Rl^d$ as a subset of $\Cm^d$ and we denote by
$\langle\cdot,\cdot\rangle$ the standard inner product of $\Cm^d$. We
write $A\succ 0$ to specify a positive-definite Hermitian matrix
$A\in\Cm^{d\times d}$.

Fix matrices $S$, $\Sigma_o$, $L$, $U$, $R$, and $V$ in $\Rl^{d\times
  d}$ with $\rho(S)<1$, $\Sigma_o\succ 0$, and $L$, $U$, and $R$
symmetric. According to (\ref{evol_eq}) and (\ref{def:W}), they define
a centered stable Gauss-Markov chain $\X$ and a quadratic functional
$W_N$ for each $N\ge 1$. For every $\lambda\in\Rl$ and
$\theta\in[0,2\pi]$, we make use of $S$, $U$, and $V$ to construct the
Hermitian matrix
\begin{equation}
F_\lambda(\theta):=\big(I-S^\top e^{\mathrm{i}\theta}\big)\big(I-Se^{-\mathrm{i}\theta}\big)-\lambda\big(U+V e^{-\mathrm{i}\theta}+V^\top e^{\mathrm{i}\theta}\big)\in\Cm^{d\times d},
\label{def:symbol}
\end{equation}
and we set
\begin{equation}
  f_\lambda:=\infp_{\theta\in[0,2\pi]}\inf_{\substack{z\in\Cm^d\\z\ne 0}}\bigg\{\frac{\langle z,F_\lambda(\theta)z\rangle}
  {\langle z,z\rangle}\bigg\}.
\label{def:infr}
\end{equation}
As $f_\lambda$ bounds the spectrum of $F_\lambda(\theta)$ from below
for all $\theta$, if $f_\lambda>0$, then the functions that map
$\theta\in[0,2\pi]$ in $\ln\det F_\lambda(\theta)$ and
$F^{-1}_\lambda(\theta)$ are well-defined and continuous. Thus, for
each $\lambda\in\Rl$ such that $f_\lambda>0$ we can introduce the
integrals
\begin{equation}
\varphi(\lambda):=-\frac{1}{4\pi}\int_0^{2\pi} \ln\det F_\lambda(\theta) \,d\theta
\label{def:phi}
\end{equation}
and
\begin{equation}
  \Phi_\lambda(n):=\frac{1}{2\pi}\int_0^{2\pi} F_\lambda^{-1}(\theta)e^{-\mathrm{i}n\theta}d\theta
\label{def:Fourierc}
\end{equation}
with $n\in\mathbb{Z}$. It will turn out that $\varphi(\lambda)$ is the
value at $\lambda$ of the cumulant generating function of $W_N$ in the
large $N$ limit: $\lim_{N\uparrow\infty}(1/N)\ln\Ex[e^{\lambda
    W_N}]=\varphi(\lambda)$. The boundary terms described by the
matrices $\Sigma_o$, $L$, and $R$ determine the effective domain of
the asymptotic cumulant generating function, and in order to specify
this domain we need the following technical lemma. The proof is
reported in Section \ref{sec:proofLDP}.
\begin{lemma}
\label{lem:tech}
Let $\lambda\in\Rl$ be such that $f_\lambda>0$. The following
conclusions hold:
\begin{enumerate}
\item $H_\lambda:=I+(S+\lambda V)\Phi_\lambda(1)\in\Cm^{d\times d}$ is
  invertible, and the matrix
  \begin{equation*}
\mathcal{L}_\lambda:=\displaystyle{\Sigma_o^{-1}+S^\top S-\lambda L-(S^\top+\lambda V^\top)\Phi_\lambda(0)H_\lambda^{-1}(S+\lambda V)}\in\Cm^{d\times d}
\end{equation*}
is Hermitian;
\item $K_\lambda:=I+\Phi_\lambda(1)(S+\lambda V)\in\Cm^{d\times d}$
  is invertible, and the matrix
  \begin{equation*}
\mathcal{R}_\lambda:=\displaystyle{I-\lambda R-(S+\lambda V)K_\lambda^{-1}\Phi_\lambda(0)(S^\top+\lambda V^\top)}\in\Cm^{d\times d}
\end{equation*}
is Hermitian.
\end{enumerate}
\end{lemma}

Lemma \ref{lem:tech} states that the matrices $\mathcal{L}_\lambda$
and $\mathcal{R}_\lambda$ are well-defined and Hermitian when
$\lambda\in\Rl$ satisfies $f_\lambda>0$. It makes then sense to
consider the extended real numbers
\begin{equation}
  \lambda_-:=\inf\Big\{\lambda\in\Rl~:~f_\lambda>0,~\mathcal{L}_\lambda\succ 0, \mbox{ and }\mathcal{R}_\lambda\succ 0\Big\}
\label{def:lambdam}
\end{equation}
and
\begin{equation}
  \lambda_+:=\sup\Big\{\lambda\in\Rl~:~f_\lambda>0,~\mathcal{L}_\lambda\succ 0, \mbox{ and }\mathcal{R}_\lambda\succ 0\Big\}.
\label{def:lambdap}
\end{equation}
We are now in the position to present the first main result of the
paper, which establishes a large deviation principle for $W_N/N$ and
is proved in Section \ref{sec:proofLDP} via a time-dependent change of
measure.
\begin{theorem}
  \label{main}
  The following conclusions hold:
  \begin{enumerate}
  \item $\lambda_-<0<\lambda_+$ and the convex function $I$ that maps
    $w\in\Rl$ in
    $I(w):=\sup_{\lambda\in(\lambda_-,\lambda_+)}\{w\lambda-\varphi(\lambda)\}$
    has compact level sets;
 \item the quadratic functional $W_N/N$ associated with the stable
   Gauss-Markov chain $\X$ satisfies a large deviation principle with
   the rate function $I$.
\end{enumerate}
\end{theorem}

Theorem \ref{main} outperforms the G\"artner-Ellis theorem, which
requires that the asymptotic cumulant generating function exists and
defines an essentially smooth, lower semicontinuous function
\cite{Frank1,Dembo}. In Section \ref{sec:proofLDP} we shall prove that
$\lim_{N\uparrow\infty}(1/N)\ln\Ex[e^{\lambda W_N}]=\varphi(\lambda)$
if $\lambda\in(\lambda_-,\lambda_+)$ and
$\lim_{N\uparrow\infty}(1/N)\ln\Ex[e^{\lambda W_N}]=+\infty$ if
$\lambda\notin\overline{(\lambda_-,\lambda_+)}$.  We shall also verify
that the function $\varphi$ that maps
$\lambda\in(\lambda_-,\lambda_+)$ in $\varphi(\lambda)$ is convex and
differentiable, so that the limits
$\lim_{\lambda\downarrow\lambda_-}\varphi(\lambda)=:\varphi_-$,
$\lim_{\lambda\uparrow\lambda_+}\varphi(\lambda)=:\varphi_+$,
$\lim_{\lambda\downarrow\lambda_-}\varphi'(\lambda)=:d_-$, and
$\lim_{\lambda\uparrow\lambda_+}\varphi'(\lambda)=:d_+$ exist.  If
even the limit $\lim_{N\uparrow\infty}(1/N)\ln\Ex[e^{\lambda W_N}]$
existed for all $\lambda\in\Rl$ and defined a lower semicontinuous
function as demanded by the G\"artner-Ellis theorem, what is generally
missing to guarantee essentially smoothness of the asymptotic cumulant
generating function is the steepness of $\varphi$, i.e.\ the property
that $d_-=-\infty$ if $\lambda_->-\infty$ and $d_+=+\infty$ if
$\lambda_+<+\infty$. The lack of steepness produces affine stretches
in the graph of the rate function. In fact, if $\lambda_->-\infty$ and
$d_->-\infty$, then $I(w)=w\lambda_--\varphi_-$ for all $w<
d_-$. Notice that $\varphi_-$ is finite in this case since
$\varphi(\lambda)\le\varphi(0)+\varphi'(\lambda)\lambda=\varphi'(\lambda)\lambda$
for all $\lambda\in(\lambda_-,\lambda_+)$ by convexity, which gives
$\varphi_-\le d_-\lambda_-$ by sending $\lambda$ to $\lambda_-$.
Similarly, $I(w)=w\lambda_+-\varphi_+$ for all $w>d_+$ with
$\varphi_+$ finite if $\lambda_+<+\infty$ and $d_+<+\infty$.  The
following example involving a quadratic functional of a
one-dimensional stable Gauss-Markov chain demonstrates the presence of
affine stretches.
\begin{example}
  Fix $s\in\Rl$ such that $|s|<1$ and consider the one-dimensional
  autoregressive model $X_{n+1}=sX_n+G_n$ for $n\ge 1$. The large
  fluctuations of the quadratic functional $W_N:=\sum_{n=1}^N X_n^2$
  have been already characterized for the non-stationary case $X_1:=0$
  \cite{Bryc1993} and for the stationary centered case corresponding
  to $\Sigma_o:=(1-s^2)^{-1}$ \cite{Dembo_Sol}. We can use our theory
  to investigate centered non-stationary situations with general
  initial variance $\Sigma_o>0$. In this example $S:=s$, $L:=2$,
  $U:=2$, $R:=2$, and $V:=0$. For all $\lambda$ and $\theta$ we find
\begin{equation*}
F_\lambda(\theta)=1+s^2-2\lambda-2s\cos(\theta),
\end{equation*}
so that $f_\lambda=1+s^2-2\lambda-2|s|$. If $f_\lambda>0$,
i.e.\ $2\lambda<(1-|s|)^2$, then easy calculations yield
\begin{equation*}
\varphi(\lambda)=-\frac{1}{2}\ln\frac{1+s^2-2\lambda+\sqrt{(1+s^2-2\lambda)^2-4s^2}}{2},
\end{equation*}
\begin{equation*}
\mathcal{L}_\lambda=\Sigma_o^{-1}+\frac{s^2-1-2\lambda+\sqrt{(1+s^2-2\lambda)^2-4s^2}}{2},
\end{equation*}
and
\begin{equation*}
\mathcal{R}_\lambda=\frac{1-s^2-2\lambda+\sqrt{(1+s^2-2\lambda)^2-4s^2}}{2}>0.
\end{equation*}
The quantities $\mathcal{L}_\lambda$ and $\mathcal{R}_\lambda$ are
defined by Lemma \ref{lem:tech}.  According to (\ref{def:lambdam}) and
(\ref{def:lambdap}), we have $\lambda_-=-\infty$ and
$2\lambda_+=(1-|s|)^2$ if $\Sigma_o^{-1}\ge 1-|s|$, and
$\lambda_-=-\infty$ and
$2\lambda_+=(\Sigma_o^{-1}-1+s^2)/(1-\Sigma_o)$ if
$\Sigma_o^{-1}<1-|s|$. In the former case $\varphi$ is steep, whereas
steepness is missing in the latter case where
$d_+=1/\sqrt{(1+s^2-2\lambda_+)^2-4s^2}<+\infty$.  If
$\Sigma_o^{-1}\ge 1-|s|$, then the rate function is
\begin{equation*}
  I(w)=J(w):=\begin{cases}
  +\infty & \mbox{if }w\le 0,\\
  \frac{1}{2}(1+s^2)w-\frac{1}{2}\ln(2w)-\frac{1}{2}\sqrt{1+(2sw)^2}+\frac{1}{2}\ln[1+\sqrt{1+(2sw)^2}] & \mbox{if } w>0.
  \end{cases}
\end{equation*}
If $\Sigma_o^{-1}<1-|s|$, then the rate function reads
\begin{equation*}
  I(w)=\begin{cases}
  J(w) & \mbox{if } w<d_+,\\
  w\lambda_+-\varphi_+  & \mbox{if } w\ge d_+.
  \end{cases}
\end{equation*}
\end{example}

As $W_N/N=e_N$ for all $N\ge 1$ when $L:=I-\Sigma_o^{-1}-S^\top S$,
$U:=0$, $R:=\Sigma_o^{-1}+S^\top S-I$, and $V:=S-S^\top$, Theorem
\ref{main} immediately shows that the entropy production rate $e_N$
satisfies a large deviation principle. The Hermitian matrix
$F_\lambda(\theta)$ corresponding to $e_N$ reads for each
$\lambda\in\Rl$ and $\theta\in[0,2\pi]$
\begin{align}
  \nonumber
  F_\lambda(\theta)&=\big(I-S^\top e^{\mathrm{i}\theta}\big)\big(I-S e^{-\mathrm{i}\theta}\big)+2\mathrm{i}\lambda\big(S-S^\top\big)\sin\theta\\
  &=I+S^\top S-(S+S^\top)\cos\theta+\mathrm{i}(2\lambda+1)\big(S-S^\top\big)\sin\theta.
  \label{F_e}
\end{align}
The second main result of the paper, whose proof is reported in
Section \ref{sec:simmetria}, confirms the Gallavotti-Cohen
symmetry. This symmetry comes from the manifest relationship
$F_{-\lambda-1}(\theta)=F_\lambda(2\pi-\theta)$.
\begin{theorem}
\label{LDP_eN}
The following conclusions hold:
\begin{enumerate}
\item the entropy production rate $e_N$ of the stable Gauss-Markov
  chain $\X$ satisfies a large deviation principle with the convex
  rate function $I$;
\item $\lambda_-=-\lambda_+-1$ and $I(-w)=I(w)+w$ for all $w\in\Rl$.
\end{enumerate}
\end{theorem}

If the drift matrix $S$ is symmetric and $\Sigma_o=\Sigma_s$, then the
process $\X$ is reversible and $e_N=0$ for all $N\ge 1$. The following
example shows that there is entropy production when $S$ is symmetric
but $\X$ is not stationary.
\begin{example}
  Assume that the drift matrix $S$ is symmetric. We have
  $\Sigma_s=(I-S^2)^{-1}$ and formula (\ref{F_e}) gives
  $F_\lambda(\theta)=\big(I-S e^{\mathrm{i}\theta}\big)\big(I-S
  e^{-\mathrm{i}\theta}\big)$ for every $\lambda$ and $\theta$. One
  can easily verify that $f_\lambda=[1-\rho(S)]^2>0$ and
  $\varphi(\lambda)=0$ for all $\lambda\in\Rl$. Starting from the
  identity $(I-Se^{\pm\mathrm{i}\theta})^{-1}=\sum_{k\ge 0}
  S^ke^{\pm\mathrm{i}k\theta}$ as $\rho(S)<1$, one can then deduce
  that for all $\lambda\in\Rl$
\begin{equation*}
\mathcal{L}_\lambda=\mathcal{R}_{-\lambda-1}=(\lambda+1)\Sigma_o^{-1}-\lambda\Sigma_s^{-1}.
\end{equation*}
Fix $\Sigma_o\succ 0$ different from $\Sigma_s$ and set
$\Delta:=(\Sigma_s-\Sigma_o)(\Sigma_s+\Sigma_o)^{-1}$. We claim that
the spectral radius $\rho(\Delta)$ of $\Delta$ is strictly positive
and that
\begin{equation}
    \lambda_\pm=\frac{1}{2}\bigg[-1\pm\frac{1}{\rho(\Delta)}\bigg].
\label{example_sym}
\end{equation}
The entropy production rate satisfies a large deviation principle with
the rate function
\begin{equation*}
  I(w)=\begin{cases}
  w\lambda_- & \mbox{if } w<0,\\
  w\lambda_+  & \mbox{if } w\ge 0.
  \end{cases}
\end{equation*}
To prove (\ref{example_sym}), let $A\in\Rl^{d\times d}$ be an
invertible matrix such that
$(1/2)(\Sigma_o^{-1}+\Sigma_s^{-1})=AA^\top$ and set
$B:=(1/2)A^{-1}(\Sigma_o^{-1}-\Sigma_s^{-1})(A^\top)^{-1}$.  The
matrix $A$ exists since $\Sigma_s\succ 0$ and $\Sigma_o\succ 0$, and
the spectral radius $\rho(B)$ of the symmetric matrix $B$ is strictly
positive since $\Sigma_o\ne\Sigma_s$. Similarity transformations show
that $\rho(B)=\rho(\Delta)$.  We have
$\mathcal{L}_\lambda=A[I+(2\lambda+1)B]A^\top\succ 0$ and
$\mathcal{R}_\lambda=A[I-(2\lambda+1)B]A^\top\succ 0$ if and only if
$|2\lambda+1|\rho(B)<1$. Thus, $(2\lambda_\pm+1)\rho(B)=\pm 1$.
\end{example}

\subsection{Entropy production with a normal drift matrix}
\label{sec:normal}

Analyzing the role of the conditions $\mathcal{L}_\lambda\succ 0$ and
$\mathcal{R}_\lambda\succ 0$ in determining those $\lambda\in\Rl$ for
which $\lim_{N\uparrow\infty}(1/N)\ln\Ex[e^{\lambda
    W_N}]=\varphi(\lambda)$ is a difficult task. We stress that the
satisfiability of these conditions shapes the effective domain
$(\lambda_-,\lambda_+)$ of the asymptotic cumulant generating function
of $W_N$. Now our interest is in the entropy production
$W_N:=Ne_N$. Computer simulations suggest that, in the stationary case
$\Sigma_o=\Sigma_s$, the Hermitian matrices $\mathcal{L}_\lambda$ and
$\mathcal{R}_\lambda$ associated with $Ne_N$ are automatically
positive-definite for the values of $\lambda$ that satisfy the primary
constraint $f_\lambda>0$. If this is true in general, then we will
conclude that $\lambda_-=\inf\{\lambda\in\Rl:f_\lambda>0\}$ and
$\lambda_+=\sup\{\lambda\in\Rl:f_\lambda>0\}$ when
$\Sigma_o=\Sigma_s$. While we leave this general problem as an open
question, we verify the conjecture
$\lambda_-=\inf\{\lambda\in\Rl:f_\lambda>0\}$ and
$\lambda_+=\sup\{\lambda\in\Rl:f_\lambda>0\}$ for a stationary stable
Gauss-Markov chain $\X$ with normal drift matrix $S$. Then, here we
assume that $S^\top S=SS^\top$. This case is very special because it
allows for explicit results.  We point out that large deviation
principles have been recently established for the entropy production
rate of stationary stable Ornstein-Uhlenbeck processes with normal
drift matrix \cite{Jaksic1,Jaksic2,Budhiraja2021}. In particular,
Budhiraja, Chen, and Xu \cite{Budhiraja2021} have exhibited explicitly
the rate function, posing the question of whether the same could have
been done for the discrete-time autoregressive model. Our work gives
an affirmative answer to their question, and indeed we provide a large
deviation principle for any drift matrix.

Dealing with a normal drift matrix in the problem of entropy
production basically means dealing with a diagonal drift matrix. In
fact, normality of $S$ implies that there exists a unitary matrix
$\Gamma\in\Cm^{d\times d}$ such that $\Gamma S\Gamma^{-1}$ and $\Gamma
S^\top\Gamma^{-1}=(\Gamma S\Gamma^{-1})^\dagger$ are both diagonal.
Let $\alpha_k+\mathrm{i}\beta_k$ be the $k$th element of the diagonal
of $\Gamma S\Gamma^{-1}$, with $\alpha_k$ and $\beta_k$ real numbers,
and notice that the stability hypothesis $\rho(S)<1$ requires that
$\alpha_k^2+\beta_k^2<1$ as $\alpha_k+\mathrm{i}\beta_k$ obviously is
an eigenvalue of $S$. We suppose that $\beta_k\ne 0$ for some $k$ in
order to not to fall again in the class of symmetric drift matrices.
According to (\ref{F_e}), $\Gamma F_\lambda(\theta)\Gamma^{-1}$ is
diagonal for all $\lambda\in\Rl$ and $\theta\in[0,2\pi]$, and the
$k$th element of the diagonal of $\Gamma F_\lambda(\theta)\Gamma^{-1}$
reads
\begin{equation*}
  1+\alpha_k^2+\beta_k^2-2\alpha_k\cos\theta-2\beta_k(2\lambda+1)\sin\theta=\big(1+\alpha_k^2+\beta_k^2\big)\big[1-\varrho_k\cos(\theta-\vartheta_k)\big]
\end{equation*}
with
\begin{equation*}
\varrho_k:=2\frac{\sqrt{\alpha_k^2+(2\lambda+1)^2\beta_k^2}}{1+\alpha_k^2+\beta_k^2}\ge 0
\end{equation*}
and
\begin{equation*}
  \vartheta_k:=\arctan\bigg(\frac{\beta_k+2\lambda\beta_k}{\alpha_k}\bigg).
\end{equation*}
We omit to indicate the dependence of $\varrho_k$ and $\vartheta_k$ on
$\lambda$ for simplicity. We have
\begin{equation*}
  f_\lambda=\inf_{\theta\in[0,2\pi]}\min_{1\le k\le d}\Big\{\big(1+\alpha_k^2+\beta_k^2\big)\big[1-\varrho_k\cos(\theta-\vartheta_k)\big]\Big\}
  =\min_{1\le k\le d}\Big\{\big(1+\alpha_k^2+\beta_k^2\big)\big(1-\varrho_k\big)\Big\},
\end{equation*}
so that the condition $f_\lambda>0$ on $\lambda$ becomes $\max_{1\le
  k\le d}\{\varrho_k\}<1$. If $\max_{1\le k\le d}\{\varrho_k\}<1$,
then we find from (\ref{def:phi})
\begin{align}
  \nonumber
  \varphi(\lambda)&=-\frac{1}{4\pi}\sum_{k=1}^d\int_0^{2\pi}\ln \Big\{\big(1+\alpha_k^2+\beta_k^2\big)\big[1-\varrho_k\cos(\theta-\vartheta_k)\big]\Big\}d\theta\\
  \nonumber
  &=-\frac{1}{4\pi}\sum_{k=1}^d\int_0^{2\pi}\ln \big(1-\varrho_k\cos\theta\big)d\theta-\frac{1}{2}\sum_{k=1}^d\ln\big(1+\alpha_k^2+\beta_k^2\big)\\
  \nonumber
  &=-\frac{1}{2}\sum_{k=1}^d \ln\frac{1+\sqrt{1-\varrho_k^2}}{2}-\frac{1}{2}\sum_{k=1}^d\ln\big(1+\alpha_k^2+\beta_k^2\big).
\end{align}
For each $n\in\mathbb{Z}$, the matrix
$\Gamma\Phi_\lambda(n)\Gamma^{-1}$ defined by (\ref{def:Fourierc}) is
diagonal with $k$th diagonal element equal to
\begin{align}
  \nonumber
  \frac{1}{2\pi}\int_0^{2\pi}\frac{e^{-\mathrm{i}n\theta}d\theta}{(1+\alpha_k^2+\beta_k^2)[1-\varrho_k\cos(\theta-\vartheta_k)]}
  &=\frac{e^{-\mathrm{i}n\vartheta_k}}{1+\alpha_k^2+\beta_k^2}\frac{1}{2\pi}\int_0^{2\pi}\frac{\cos(n\theta)d\theta}{1-\varrho_k\cos\theta}\\
  \nonumber
  &=\frac{1}{1+\alpha_k^2+\beta_k^2}\frac{e^{-\mathrm{i}n\vartheta_k}}{\sqrt{1-\varrho_k^2}}\bigg(\frac{1-\sqrt{1-\varrho_k^2}}{\varrho_k}\bigg)^{|n|}.
\end{align}
Under the constraint $\max_{1\le k\le d}\{\varrho_k\}<1$, the matrices
$\mathcal{L}_\lambda$ and $\mathcal{R}_\lambda$ associated by Lemma
\ref{lem:tech} with $L:=I-\Sigma_o^{-1}-S^\top S$, $U:=0$,
$R:=\Sigma_o^{-1}+S^\top S-I$, and $V:=S-S^\top$ can be written as
\begin{equation}
  \mathcal{L}_\lambda=(\lambda+1)(\Sigma_o^{-1}-\Sigma_s^{-1})+\mathcal{M}_\lambda
\label{L_normal}
\end{equation}
and
\begin{equation}
  \mathcal{R}_\lambda=\lambda(\Sigma_s^{-1}-\Sigma_o^{-1})+\mathcal{M}_\lambda,
\label{R_normal}
\end{equation}
where $\Gamma\mathcal{M}_\lambda\Gamma^{-1}\in\Cm^{d\times d}$ is
diagonal with $k$th diagonal element given by
\begin{equation}
\frac{1-\alpha_k^2-\beta_k^2}{2}+\frac{1+\alpha_k^2+\beta_k^2}{2}\sqrt{1-\varrho_k^2}>0.
\label{eig_M_normal}
\end{equation}
To obtain (\ref{L_normal}) and (\ref{R_normal}) we have used the
facts that $\Sigma_s=(I-SS^\top)^{-1}$ and that
$\Gamma\Sigma_s^{-1}\Gamma^{-1}$ is diagonal with $k$th diagonal entry
equal to $1-\alpha_k^2-\beta_k^2$. Importantly, the Hermitian matrix
$\mathcal{M}_\lambda$ is positive-definite as demonstrated by
(\ref{eig_M_normal}).

If the chain $\X$ is stationary, i.e.\ if $\Sigma_o=\Sigma_s$, then
$\mathcal{L}_\lambda=\mathcal{M}_\lambda$ and
$\mathcal{R}_\lambda=\mathcal{M}_\lambda$ are automatically
positive-definite when $\max_{1\le k\le d}\{\varrho_k\}<1$, namely
when $f_\lambda>0$. Thus, the conjecture
$\lambda_-=\inf\{\lambda\in\Rl:f_\lambda>0\}$ and
$\lambda_+=\sup\{\lambda\in\Rl:f_\lambda>0\}$ for a stationary stable
Gauss-Markov chain is true if the drift matrix is normal. Furthermore,
in this case $\lambda_-$ and $\lambda_+$ are the smallest and the
largest values of $\lambda$ for which $\max_{1\le k\le
  d}\{\varrho_k\}=1$, which are explicitly given by the formulas
\begin{equation*}
\lambda_+=\lambda_o:=-\frac{1}{2}+\min_{1\le k\le d}\Bigg\{\sqrt{\frac{(1+\alpha_k^2+\beta_k^2)^2-4\alpha_k^2}{16\beta_k^2}}\Bigg\}
\end{equation*}
and
\begin{equation*}
\lambda_-=-\lambda_o-1.
\end{equation*}
Notice that $\lambda_o$ is finite since we are supposing that
$\beta_k\ne 0$ for some $k$.  With such $\lambda_-$ and $\lambda_+$,
the function $\varphi$ turns out to be steep in
$(\lambda_-,\lambda_+)$. Thus, for each $w\in\Rl$ there exists a
unique $\lambda\in(\lambda_-,\lambda_+)$ such that
$w=\varphi'(\lambda)$ and, as a consequence,
\begin{equation*}
I(w)=w\lambda-\varphi(\lambda).
\end{equation*}
Basically, this is the result found by Budhiraja, Chen, and Xu
\cite{Budhiraja2021} for the continuous-time model.

To conclude, let us briefly discuss what happens when the chain $\X$
is not stationary, i.e.\ when $\Sigma_o\ne\Sigma_s$. If
$\mathcal{L}_\lambda\succ 0$ and $\mathcal{R}_\lambda\succ 0$ for all
$\lambda\in(-\lambda_o-1,\lambda_o)$, then $\lambda_+=\lambda_o$ and
$\lambda_-=-\lambda_o-1$, as before, and the function $\varphi$ is
steep in $(\lambda_-,\lambda_+)$. We have $\mathcal{L}_\lambda\succ 0$
and $\mathcal{R}_\lambda\succ 0$ for all
$\lambda\in(-\lambda_o-1,\lambda_o)$ if
$\mathcal{L}_{\lambda_o}=\mathcal{R}_{-\lambda_o-1}\succ 0$ and
$\mathcal{R}_{\lambda_o}=\mathcal{L}_{-\lambda_o-1}\succ 0$ as
formulas (\ref{L_normal}) and (\ref{R_normal}) show that the functions
that map $\lambda$ in $\langle z,\mathcal{L}_\lambda z\rangle$ and
$\langle z,\mathcal{R}_\lambda z\rangle$ are concave for any given
$z\in\Cm^d$.  If, on the contrary, there exists
$\lambda\in(-\lambda_o-1,\lambda_o)$ such that
$\mathcal{L}_\lambda\nsucc 0$ or $\mathcal{R}_\lambda\nsucc 0$, then
$\lambda_+<\lambda_o$ and $\lambda_-=-\lambda_+-1>-\lambda_o-1$.  For
example, this occurs for $\Sigma_o=\sigma I$ with a sufficiently small
$\sigma>0$.  In this case $\varphi$ is not steep in
$(\lambda_-,\lambda_+)$ and the rate function at $w\in\Rl$ has the
value
\begin{equation*}
  I(w)=\begin{cases}
  w\lambda_--\varphi_- & \mbox{if }w\le d_-,\\
  w\lambda-\varphi(\lambda) & \mbox{if } d_-<w<d_+, \\
  w\lambda_+-\varphi_+ & \mbox{if }w\ge d_+,
  \end{cases}
\end{equation*}
where, regarding the case $d_-<w<d_+$, $\lambda$ is the unique real
number in $(\lambda_-,\lambda_+)$ that satisfies
$w=\varphi'(\lambda)$. Breaking stationarity can then involve affine
stretches in the graph of the entropy production rate function.

\section{Proof of Lemma \ref{lem:tech} and Theorem \ref{main}}
\label{sec:proofLDP}

In this section we prove Theorem \ref{main}, which states the large
deviation principle for the quadratic functional $W_N$ defined by
(\ref{def:W}). The proof of Theorem \ref{main} is based on a
time-dependent change of measure and requires at first to study the
asymptotics of the cumulant generating function of $W_N$ as $N$ goes
to infinity. In turn, this asks for investigation of Hermitian block
tridiagonal quasi-Toeplitz matrices that differ from Hermitian block
tridiagonal Toeplitz matrices by the first and last diagonal
blocks. In Section \ref{sec:QT} we introduce these matrices and
characterize their positive definiteness property and their
determinant. Section \ref{sec:cum} uses the theory of Section
\ref{sec:QT} to compute the scaled cumulant generating function of
$W_N$ in the large $N$ limit.  The upper large deviation bound for
closed sets is proved in Section \ref{sec:upper}. Finally, the lower
large deviation bound for open sets is established in Section
\ref{sec:lower}. Along the way we shall also verify Lemma
\ref{lem:tech}.

As we have already said, we regard $\Rl^d$ as a subset of $\Cm^d$. We
denote by $\langle\cdot,\cdot\rangle$ the standard inner product of
$\Cm^d$ and by $\|\cdot\|$ the induced norm.  If
$\zeta=(\zeta_1,\ldots,\zeta_N)$ and $z=(z_1,\ldots, z_N)$ are two
vectors in $(\Cm^d)^N$, $N$ being a positive integer, we understand
that $\langle \zeta,z\rangle:=\sum_{n=1}^N\langle \zeta_n,z_n\rangle$
and $\|z\|^2:=\sum_{n=1}^N\langle
z_n,z_n\rangle=\sum_{n=1}^N\|z_n\|^2$. For positive integers $M$ and
$N$, $\mathsf{BL}_{M,N}$ is the set of complex block matrices with
$M\times N$ square blocks of size $d$. For any
$\mathsf{A}\in\mathsf{BL}_{N,N}$, $\|\mathsf{A}\|$ is the operator
norm of $\mathsf{A}$ induced by the norm of $(\Cm^d)^N$:
\begin{equation*}
\|\mathsf{A}\|:=\sup_{\substack{z\in(\Cm^d)^N\\z\ne 0}}\bigg\{\frac{\|\mathsf{A}z\|}{\|z\|}\bigg\}.
\end{equation*}
Given a Hermitian matrix $\mathsf{A}\in\mathsf{BL}_{N,N}$, we denote
by $r(\mathsf{A})$ the infimum of the Rayleigh quotient of a
$\mathsf{A}$, that is the smallest eigenvalue of $\mathsf{A}$:
\begin{equation*}
r(\mathsf{A}):=\inf_{\substack{z\in(\Cm^d)^N\\z\ne 0}}\bigg\{\frac{\langle z,\mathsf{A}z\rangle}{\langle z,z\rangle}\bigg\}.
\end{equation*}
If a Hermitian matrix $\mathsf{A}\in\mathsf{BL}_{N,N}$ is
positive-definite we write $\mathsf{A}\succ 0$. We have
$\mathsf{A}\succ 0$ if and only if $r(\mathsf{A})>0$.

\subsection{On Hermitian block tridiagonal quasi-Toeplitz matrices}
\label{sec:QT}

The coefficient matrix of the quadratic functional $W_N$ is an element
from a sequence of Hermitian block tridiagonal matrices in the
following class.
\begin{definition}
  \label{defQ}
  A sequence of matrices $\Qseq$, with
  $\mathsf{Q}_N\in\mathsf{BL}_{N+2,N+2}$ for each $N$, is a Hermitian
  block tridiagonal quasi-Toeplitz (HQT) matrix sequence if there
  exist four square matrices $A$, $D$, $B$, and $E$ of size $d$, with
  $A$, $D$, and $B$ Hermitian, such that for all $N\ge 1$
  \begin{equation*}
  \mathsf{Q}_N=\begin{pmatrix}
  A & E^\dagger &  & & \\
    E & D & \ddots &  &\\
      & \ddots & \ddots & \ddots  & \\
     & & \ddots & D & E^\dagger\\
     & &  & E & B
  \end{pmatrix}.
  \end{equation*}
\end{definition}

In this section we characterize asymptotic positive definiteness and
asymptotic determinants of matrices from a HQT matrix sequence,
postponing the most technical proofs in the appendices. We stress that
a HQT matrix sequence is bounded in the following sense, which is
demonstrated in Appendix \ref{proof:norm}.
\begin{lemma}
  \label{lem:norm}
Let $\Qseq$ be a HQT matrix sequence with $A$, $D$, $B$, and $E$ as in
the above definition. Then, for each $N\ge 1$
\begin{equation*}
\|\mathsf{Q}_N\|\le\sqrt{2\|A\|^2+3\|D\|^2+2\|B\|^2+6\|E\|^2}.
\end{equation*}
\end{lemma}

In order to deal with a HQT matrix sequence $\Qseq$, it is convenient
to isolate the \textit{bulk matrix} $\mathsf{T}_N\in\mathsf{BL}_{N,N}$
of $\mathsf{Q}_N$ defined by
 \begin{equation}
  \mathsf{T}_N:=\begin{pmatrix}
    D & E^\dagger &  &   \\
    E & \ddots & \ddots &   \\
       &  \ddots   &  \ddots & E^\dagger\\
     &  &  E & D
  \end{pmatrix}.
 \end{equation}
The bulk matrix $\mathsf{T}_N$ is a Hermitian block tridiagonal Toeplitz
matrix, which allows $\mathsf{Q}_N$ to be written as
\begin{equation}
  \mathsf{Q}_N=\begin{pmatrix}
    A & E^\dagger\mathsf{C}^\dagger &  0 \\
    \mathsf{C}E & \mathsf{T}_N & \mathsf{R}^\dagger E^\dagger \\
    0 & E\mathsf{R} & B
  \end{pmatrix},
  \label{QconT}
\end{equation}
where
\begin{equation}
  \mathsf{C}:=\begin{pmatrix}
    I \\
    0\\
    \vdots \\
    0
  \end{pmatrix}
  \in\mathsf{BL}_{N,1}
  \label{defCmat}
\end{equation}
and
\begin{equation}
  \mathsf{R}:=\begin{pmatrix}
   0 &\cdots& 0 & I
  \end{pmatrix}
  \in\mathsf{BL}_{1,N}.
\label{defRmat}
\end{equation}
 When $\mathsf{T}_N$ is invertible we introduce the \textit{boundary
   matrix} $\mathsf{S}_N\in\mathsf{BL}_{2,2}$ of $\mathsf{Q}_N$ defined
 by
\begin{equation}
\mathsf{S}_N:=\begin{pmatrix}
    A-E^\dagger \mathsf{C}^\dagger \mathsf{T}_N^{-1}\mathsf{C}E &  -E^\dagger\mathsf{C}^\dagger\mathsf{T}_N^{-1}\mathsf{R}^\dagger E^\dagger \\[0.3em]
    -E\mathsf{R}\mathsf{T}_N^{-1}\mathsf{C}E &  B-E\mathsf{R}\mathsf{T}_N^{-1}\mathsf{R}^\dagger E^\dagger
\end{pmatrix}.
\label{QS}
\end{equation}
Manifestly, $\mathsf{S}_N$ is a Hermitian matrix.  The following lemma
relates the positive definiteness and the determinant of
$\mathsf{Q}_N$ to those of the bulk matrix $\mathsf{T}_N$ and the
boundary matrix $\mathsf{S}_N$. The proof is reported in Appendix
\ref{proof:primo}.
\begin{lemma}
  \label{primo}
Let $\Qseq$ be a HQT matrix sequence with bulk matrices $\mathsf{T}_N$
and boundary matrices $\mathsf{S}_N$. The following
conclusions hold for any $N\ge 1$:
  \begin{enumerate}
  \item if $r(\mathsf{Q}_N)\ge q$ for some real number $q>0$, then
    $r(\mathsf{T}_N)\ge q$ (which implies that
    $\mathsf{T}_N$ is invertible) and $r(\mathsf{S}_N)\ge q$;
  \item if $\mathsf{T}_N\succ \mathsf{0}$ (which implies that
    $\mathsf{T}_N$ is invertible) and $\mathsf{S}_N\succ\mathsf{0}$,
    then $\mathsf{Q}_N\succ \mathsf{0}$ and
  \begin{equation*}
    \ln\det\mathsf{Q}_N=\ln\det \mathsf{T}_N+\ln\det \mathsf{S}_N.
  \end{equation*}
  \end{enumerate}
\end{lemma}

We now examine the bulk matrices. For each $\theta\in[0,2\pi]$, let
$F(\theta)\in\Cm^{d\times d}$ be a Hermitian matrix defined by
\begin{equation*}
  F(\theta):=E e^{-\mathrm{i}\theta}+D+E^\dagger e^{\mathrm{i}\theta},
\end{equation*}
$D$ and $E$ being the matrices that identify the bulk matrix
$\mathsf{T}_N$ of $\mathsf{Q}_N$.  In the theory of block Toeplitz
matrices \cite{Toeplitz}, the function $F$ that maps $\theta$ in
$F(\theta)$ is called the \textit{symbol} of the matrices
$\mathsf{T}_N$. We shall equally call $F$ the symbol of $\mathsf{T}_N$
or the symbol of $\mathsf{Q}_N$.  The blocks of $\mathsf{T}_N$ are
related to the Fourier coefficients of the symbol $F$. In fact, for
all $N\ge 1$, $\zeta=(\zeta_1,\ldots,\zeta_N)\in(\Cm^d)^N$, and
$z=(z_1,\ldots,z_N)\in(\Cm^d)^N$ we have
  \begin{align}
    \nonumber
    \langle \zeta,\mathsf{T}_Nz\rangle
  &=\sum_{m=1}^N\sum_{n=1}^N \Bigg\langle \zeta_m,\frac{1}{2\pi}\int_0^{2\pi}F(\theta)e^{\mathrm{i}(m-n)\theta}d\theta \, z_n\Bigg\rangle\\
  &=\frac{1}{2\pi}\int_0^{2\pi}\Bigg\langle \sum_{n=1}^N\zeta_ne^{-\mathrm{i}n\theta},F(\theta)\sum_{n=1}^Nz_ne^{-\mathrm{i}n\theta}\Bigg\rangle \,d\theta .
  \label{TFrel}
\end{align}
The following lemma describes the positive definiteness and the
determinant of the bulk matrices $\mathsf{T}_N$. The proof is provided
in Appendix \ref{proof:secondo}. We stress that if
$\inf_{\theta\in[0,2\pi]}\{r(F(\theta))\}>0$, then the function that
associates $\theta\in[0,2\pi]$ with $\ln\det F(\theta)$ is
well-defined and continuous.
\begin{lemma}
  \label{secondo}
  Let $\mathsf{T}_N$ be the bulk matrices of a HQT matrix sequence
  with symbol $F$.  The following conclusions hold:
 \begin{enumerate}
\item if there exists a diverging sequence $\{N_k\}_{k\ge 0}$ of
  positive integers such that $r(\mathsf{T}_{N_k})\ge t$ for all $k\ge
  0$ with some $t\in\Rl$, then $r(\mathsf{T}_N)\ge t$ for all $N\ge
  1$;
  \item $r(\mathsf{T}_N)\ge t$ for all $N\ge 1$ with some $t\in\Rl$ if
    and only if $\inf_{\theta\in[0,2\pi]}\{r(F(\theta))\}\ge t$;
\item if $\inf_{\theta\in[0,2\pi]}\{r(F(\theta))\}>0$, then 
  \begin{equation*}
    \lim_{N\uparrow\infty}\frac{1}{N}\ln\det \mathsf{T}_N=\frac{1}{2\pi}\int_0^{2\pi}\ln\det F(\theta) \,d\theta.
  \end{equation*}
 \end{enumerate}
\end{lemma}

The analysis of the boundary matrices $\mathsf{S}_N$ is based on the
possibility to determine a limit boundary matrix when $N$ is sent to
infinity. This is done by the following lemma, which is proved in
Appendix \ref{proof:terzo}. Let $A$, $D$, $B$, and $E$ as in
Definition \ref{defQ}. Set for each $n\in\mathbb{Z}$
\begin{equation*}
  \Phi(n):=\frac{1}{2\pi}\int_0^{2\pi} F^{-1}(\theta)e^{-\mathrm{i}n\theta}d\theta,
\end{equation*}
which is a well-defined matrix under the hypothesis
$\inf_{\theta\in[0,2\pi]}\{r(F(\theta))\}>0$.
\begin{lemma}
  \label{lem:terzo}
Let $\mathsf{S}_N$ be the boundary matrices of a HQT matrix sequence
with symbol $F$. Assume that
$\inf_{\theta\in[0,2\pi]}\{r(F(\theta))\}>0$.  The following
conclusions hold:
\begin{enumerate}
\item $H:=I-E\Phi(1)\in\Cm^{d\times d}$
and $K:=I-\Phi(1)E\in\Cm^{d\times d}$ are invertible;
\item the limit
  $\lim_{N\uparrow\infty}\mathsf{S}_N=:\mathsf{S}_\infty$ exists and
  $\mathsf{S}_\infty=\begin{psmallmatrix}\mathcal{L} & 0 \\ 0 &
  \mathcal{R} \end{psmallmatrix}$ with Hermitian matrices
  $\mathcal{L}$ and $\mathcal{R}$ defined, respectively, by
  \begin{equation*}
  \mathcal{L}:=A-E^\dagger \Phi(0)H^{-1}E 
  \end{equation*}
  and
  \begin{equation*}
  \mathcal{R}:=B-EK^{-1}\Phi(0)E^\dagger.
  \end{equation*}
\end{enumerate}
\end{lemma}
We call $\mathsf{S}_\infty$ the \textit{limit boundary matrix} of the
HQT matrix sequence $\Qseq$. Putting the pieces together in the
following proposition, we finally solve the positive definiteness and the
determinants of the matrices $\mathsf{Q}_N$ in the large $N$ limit.
\begin{proposition}
  \label{detQ}
Let $\Qseq$ be a HQT matrix sequence with symbol $F$ and limit
boundary matrix $\mathsf{S}_\infty=\begin{psmallmatrix}\mathcal{L} & 0
\\ 0 & \mathcal{R} \end{psmallmatrix}$. Assume that
$\inf_{\theta\in[0,2\pi]}\{r(F(\theta))\}>0$, $\mathcal{L}\succ 0$,
and $\mathcal{R}\succ 0$. Then, $\mathsf{Q}_N\succ\mathsf{0}$ for all
sufficiently large $N$ and
    \begin{equation*}
    \lim_{N\uparrow\infty}\frac{1}{N}\ln\det \mathsf{Q}_N=\frac{1}{2\pi}\int_0^{2\pi}\ln\det F(\theta) \,d\theta.
    \end{equation*}
\end{proposition}

\begin{proof}
According to Lemma \ref{secondo}, the hypothesis
$\inf_{\theta\in[0,2\pi]}\{r(F(\theta))\}>0$ gives
$\mathsf{T}_N\succ\mathsf{0}$ for all $N\ge 1$ and
 \begin{equation*}
    \lim_{N\uparrow\infty}\frac{1}{N}\ln\det \mathsf{T}_N=\frac{1}{2\pi}\int_0^{2\pi}\ln\det F(\theta) \,d\theta.
 \end{equation*}
It also shows that
$\lim_{N\uparrow\infty}\mathsf{S}_N=\mathsf{S}_\infty$ exists by Lemma
\ref{lem:terzo}. Since $\mathsf{S}_\infty\succ \mathsf{0}$ by
hypothesis, the boundary matrices $\mathsf{S}_N$ are positive definite
for all sufficiently large $N$. It follows by Lemma \ref{primo} that
$\mathsf{Q}_N\succ\mathsf{0}$ for all sufficiently large $N$ and that
\begin{equation*}
  \lim_{N\uparrow\infty}\frac{1}{N}\ln\det \mathsf{Q}_N=\lim_{N\uparrow\infty}\frac{1}{N}\ln\det \mathsf{T}_N+\lim_{N\uparrow\infty}\frac{1}{N}\ln\det \mathsf{S}_N
  =\frac{1}{2\pi}\int_0^{2\pi}\ln\det F(\theta) \,d\theta.
\qedhere
\end{equation*}
\end{proof}

\subsection{The cumulant generating function of $W_N$}
\label{sec:cum}

Let us move to the stable Gauss-Markov chain $\X$ and the quadratic
functional $W_N$.  According to (\ref{evol_eq}), for each $N\ge 1$ the
law of $(X_1,\ldots,X_{N+2})$ is the multivariate Gaussian
distribution that at $x=(x_1,\ldots,x_{N+2})\in(\Rl^d)^{N+2}$ has
probability density
\begin{equation*}
  \frac{e^{-\frac{1}{2}\langle x, \Sigma_N^{-1} x\rangle}}{\sqrt{(2\pi)^{(N+2)d}\det\Sigma_N}}
  =\frac{e^{-\frac{1}{2}\langle x_1, \Sigma_o^{-1} x_1\rangle}}{\sqrt{(2\pi)^d\det\Sigma_o}}\prod_{n=2}^{N+2}\frac{1}{\sqrt{(2\pi)^d}}\,e^{-\frac{1}{2}\|x_n-Sx_{n-1}\|^2}.
\end{equation*}
We see that the inverse $\Sigma_N^{-1}$ of the covariance matrix
$\Sigma_N$ of $(X_1,\ldots,X_{N+2})$ is the real symmetric block
tridiagonal matrix in $\mathsf{BL}_{N+2,N+2}$ given by
\begin{equation*}
  \Sigma_N^{-1}:=\begin{pmatrix}
  \Sigma_o^{-1}+S^\top S & -S^\top &  & & \\
    -S & I+S^\top S & \ddots &  &\\
      & \ddots & \ddots & \ddots  & \\
      &  & \ddots & I+S^\top S & -S^\top\\
      &  &  & -S & I
  \end{pmatrix},
\end{equation*}
and we have $\det\Sigma_N=\det\Sigma_o$. Together with
$\Sigma_N^{-1}$, we introduce the real symmetric block tridiagonal
matrix $\mathsf{M}_N\in\mathsf{BL}_{N+2,N+2}$ defined by
\begin{equation*}
  \mathsf{M}_N:=
  \begin{pmatrix}
  L & V^\top &  & & \\
  V & U & \ddots & &\\
      & \ddots & \ddots & \ddots  & \\
      & & \ddots & U & V^\top\\
      &  &  & V & R
  \end{pmatrix}.
\end{equation*}
The matrix $\mathsf{M}_N$ allows us to express the quadratic
functional $W_{N+2}$ as $(1/2)\langle X,\mathsf{M}_NX\rangle$ with
$X:=(X_1,\ldots,X_{N+2})$. The \textit{cumulant generating function} of
$W_N$ is the function that maps $\lambda\in\Rl$ in
$(1/N)\ln\Ex[e^{\lambda W_N}]$.  We start with the following
elementary result involving Gaussian integrals.
\begin{lemma}
  \label{lem:Gaussint}
   For each $N\ge 1$ and $\lambda\in\Rl$
\begin{equation*}
\ln\Ex\big[e^{\lambda W_{N+2}}\big]=
\begin{cases}
  -\frac{1}{2}\ln\det\Sigma_o-\frac{1}{2}\ln\det\big(\Sigma_N^{-1}-\lambda\mathsf{M}_N\big) & \mbox{if}~\Sigma_N^{-1}-\lambda\mathsf{M}_N\succ \mathsf{0},\\
  +\infty & \mbox{otherwise}.
\end{cases}
\end{equation*}
\end{lemma}

We aim to investigate the asymptotics of the cumulant generating
function. According to Definition \ref{defQ}, the matrices
$\mathsf{Q}_N:=\Sigma_N^{-1}-\lambda\mathsf{M}_N\in\mathsf{BL}_{N+2,N+2}$
with some $\lambda\in\Rl$ form a HQT matrix sequence. Explicitly, we
have $A:=\Sigma_o^{-1}+S^\top S-\lambda L$, $D:=I+S^\top S-\lambda U$,
$B:=I-\lambda R$, and $E:=-S-\lambda V$.  The symbol $F_\lambda$ of
this HQT matrix sequence reads
\begin{align}
  \nonumber
  F_\lambda(\theta)&:=-\big(S+\lambda V\big)e^{-\mathrm{i}\theta}+I+S^\top S-\lambda U-\big(S^\top+\lambda V^\top\big)e^{\mathrm{i}\theta}\\
  \nonumber
  &=\big(I-S^\top e^{\mathrm{i}\theta}\big)\big(I-Se^{-\mathrm{i}\theta}\big)-\lambda\big(U+V e^{-\mathrm{i}\theta}+V^\top e^{\mathrm{i}\theta}\big)
\end{align}
for every $\theta\in[0,2\pi]$. It is exactly the matrix
(\ref{def:symbol}). According to (\ref{def:infr}), the real number
$f_\lambda$ is related to Rayleigh quotients of the symbol $F_\lambda$
by $\inf_{\theta\in[0,2\pi]}\{r(F_\lambda(\theta))\}=f_\lambda$. Lemma
\ref{lem:terzo} proves the technical Lemma \ref{lem:tech}, and the
matrices $\mathcal{L}_\lambda$ and $\mathcal{R}_\lambda$ defined by
Lemma \ref{lem:tech} enter the limit boundary matrix of $\Qseq$:
$\mathsf{S}_\infty=\begin{psmallmatrix}\mathcal{L}_\lambda & 0 \\ 0 &
\mathcal{R}_\lambda \end{psmallmatrix}$. By combining Lemma
\ref{lem:Gaussint} with Proposition \ref{detQ} we get that if
$f_\lambda>0$, $\mathcal{L}_\lambda\succ 0$, and
$\mathcal{R}_\lambda\succ 0$, then
\begin{equation}
\lim_{N\uparrow\infty}\frac{1}{N}\ln\Ex\big[e^{\lambda W_N}\big]=-\frac{1}{4\pi}\int_0^{2\pi}\ln\det F_\lambda(\theta) \,d\theta=:\varphi(\lambda),
\label{limCGF}
\end{equation}
$\varphi(\lambda)$ being the integral already defined in
(\ref{def:phi}). We want to prove here that the set
\begin{equation}
\Lambda:=\Big\{\lambda\in\Rl~:~f_\lambda>0,~\mathcal{L}_\lambda\succ 0, \mbox{ and }\mathcal{R}_\lambda\succ 0\Big\}
\label{def:Lambda_set}
\end{equation}
is an interval. Formulas (\ref{def:lambdam}) and (\ref{def:lambdap})
states that $\lambda_-=\inf\{\Lambda\}$ and
$\lambda_+=\sup\{\Lambda\}$. To begin with, we need the following
bound for $r(\Sigma_N^{-1})$, which is based on the hypothesis that
the spectral radius $\rho(S)$ of $S$ is smaller than 1 and is proved
in Appendix \ref{proof:sigma}.
\begin{lemma}
\label{lem:sigma}
There exists a real number $\sigma>0$ such that $r(\Sigma_N^{-1})\ge \sigma$ for all $N\ge 1$.
\end{lemma}
The following lemma shows that $\Lambda$ is a convex set, and hence it
is an interval.
\begin{lemma}
  \label{lem:lambdaset}
  The following limits exist and are finite:
  \begin{equation*}
    \adjustlimits\lim_{\substack{N\uparrow\infty\\ \phantom{x}}}\inf_{\substack{z\in(\Cm^d)^{N+2}\\ z\neq 0}}
    \bigg\{\frac{\langle z,\mathsf{M}_Nz\rangle}{\langle z,\Sigma_N^{-1}z\rangle}\bigg\}=:\xi_-.
  \end{equation*}
  and
  \begin{equation*}
    \adjustlimits\lim_{\substack{N\uparrow\infty\\ \phantom{x}}}\sup_{\substack{z\in(\Cm^d)^{N+2}\\ z\neq 0}}
    \bigg\{\frac{\langle z,\mathsf{M}_Nz\rangle}{\langle z,\Sigma_N^{-1}z\rangle}\bigg\}=:\xi_+.
  \end{equation*}
  If $\lambda\in\Lambda$, then $\lambda\xi_-\le 1$ and
  $\lambda\xi_+\le 1$.  If $\lambda\in\Rl$ is such that
  $\lambda\xi_-<1$ and $\lambda\xi_+<1$, then $\lambda\in\Lambda$.
\end{lemma}

\begin{proof}
Fix a real number $\lambda_o$ and set
\begin{equation*}
  \adjustlimits\liminf_{\substack{N\uparrow\infty\\ \phantom{x}}}\sup_{\substack{z\in(\Cm^d)^{N+2}\\ z\neq 0}}
  \bigg\{\lambda_o\frac{\langle z,\mathsf{M}_Nz\rangle}{\langle z,\Sigma_N^{-1}z\rangle}\bigg\}=:\xi_o.
\end{equation*}
The limit $\xi_o$ is finite since $\|\mathsf{M}_N\|\le C$ for every
$N\ge 1$ with some constant $C<+\infty$ by Lemma \ref{lem:norm} and
$\langle z,\Sigma_N^{-1}z\rangle\ge \sigma\langle z,z\rangle$ and for
all $N\ge 1$ and $z\in(\Cm^d)^{N+2}$ by Lemma \ref{lem:sigma}. Let us show
that
\begin{equation}
  \adjustlimits\limsup_{\substack{N\uparrow\infty\\ \phantom{x}}}\sup_{\substack{z\in(\Cm^d)^{N+2}\\ z\neq 0}}
  \bigg\{\lambda_o\frac{\langle z,\mathsf{M}_Nz\rangle}{\langle z,\Sigma_N^{-1}z\rangle}\bigg\}\le\xi_o.
\label{prop11}
\end{equation}
By choosing $\lambda_o=-1$ and $\lambda_o=1$, this proves that the
following limits exist and are finite:
\begin{equation*}
    \adjustlimits\lim_{\substack{N\uparrow\infty\\ \phantom{x}}}\inf_{\substack{z\in(\Cm^d)^{N+2}\\ z\neq 0}}
    \bigg\{\frac{\langle z,\mathsf{M}_Nz\rangle}{\langle z,\Sigma_N^{-1}z\rangle}\bigg\}=:\xi_-.
  \end{equation*}
  and
  \begin{equation*}
    \adjustlimits\lim_{\substack{N\uparrow\infty\\ \phantom{x}}}\sup_{\substack{z\in(\Cm^d)^{N+2}\\ z\neq 0}}
    \bigg\{\frac{\langle z,\mathsf{M}_Nz\rangle}{\langle z,\Sigma_N^{-1}z\rangle}\bigg\}=:\xi_+.
  \end{equation*}

Pick an arbitrary real number $\xi>\xi_o$ and $\epsilon>0$ such that
$\xi_o+2\epsilon\le\xi$. Consider the HQT matrix sequence $\Qseq$ with
$\mathsf{Q}_N:=\xi\Sigma_N^{-1}-\lambda_o\mathsf{M}_N\in\mathsf{BL}_{N+2,N+2}$
for all $N\ge 1$. By definition of $\xi_o$, there exists a diverging
sequence $\{N_k\}_{k\ge 0}$ of positive integers with the property
that for all $k$ and $z\in(\Cm^d)^{N_k+2}$
\begin{equation*}
\lambda_o\langle z,\mathsf{M}_{N_k}z\rangle\le (\xi-\epsilon)\langle z,\Sigma_{N_k}^{-1}z\rangle.
\end{equation*}
It follows that $r(\mathsf{Q}_{N_k})\ge\epsilon
r(\Sigma_{N_k}^{-1})\ge\epsilon \sigma>0$ for any $k\ge 0$.  Then,
part 1 of Lemma \ref{primo} tells us that
$r(\mathsf{T}_{N_k})\ge\epsilon \sigma$ and
$r(\mathsf{S}_{N_k})\ge\epsilon \sigma$ for all $k$, $\mathsf{T}_N$
being the bulk matrix of $\mathsf{Q}_N$ and $\mathsf{S}_N$ being its
boundary matrix. As a consequence, parts 1 and 2 of Lemma
\ref{secondo} give $r(\mathsf{T}_N)\ge\epsilon \sigma$ for every $N\ge
1$ and $r(F(\theta))\ge \epsilon \sigma$ for all $\theta\in[0,2\pi]$,
$F$ being the symbol of the Hermitian block Toeplitz matrices
$\mathsf{T}_N$. This way, Lemma \ref{lem:terzo} shows that
$\lim_{N\uparrow\infty}\mathsf{S}_N=\mathsf{S}_\infty$ exists and is
well-defined. Since $r(\mathsf{S}_{N_k})\ge\epsilon \sigma$ for all
$k$ we have $r(\mathsf{S}_\infty)\ge\epsilon \sigma$, so that
$r(\mathsf{S}_N)>0$ for all sufficiently large $N$. In conclusion, we
find that both $\mathsf{T}_N\succ \mathsf{0}$ and $\mathsf{S}_N\succ
\mathsf{0}$ for all sufficiently large $N$, and part 2 of Lemma
\ref{primo} ensures us that $\mathsf{Q}_N\succ \mathsf{0}$ for all
such $N$. This means that
\begin{equation}
\langle z,\mathsf{Q}_Nz\rangle=\langle z,\big(\xi\Sigma_N^{-1}-\lambda_o\mathsf{M}_N\big)z\rangle>0
\label{ratio1}
\end{equation}
for all sufficiently large $N$ and $z\in(\Cm^d)^{N+2}$.  It follows
that
\begin{equation*}
  \adjustlimits\limsup_{\substack{N\uparrow\infty\\ \phantom{x}}}\sup_{\substack{z\in(\Cm^d)^{N+2}\\ z\neq 0}}
  \bigg\{\lambda_o\frac{\langle z,\mathsf{M}_Nz\rangle}{\langle z,\Sigma_N^{-1}z\rangle}\bigg\}\le\xi,
\end{equation*}
which demonstrates (\ref{prop11}) thanks to the arbitrariness of
$\xi>\xi_o$.

Let us demonstrate now the connection between the set $\Lambda$ and
the number $\xi_-$ and $\xi_+$. Fix $\lambda\in\Lambda$ and consider
the HQT matrix sequence $\Qseq$ with matrices
$\mathsf{Q}_N:=\Sigma_N^{-1}-\lambda\mathsf{M}_N$. We already know
that this HQT matrix sequence has symbol $F_\lambda$ and limit
boundary matrix
$\mathsf{S}_\infty=\begin{psmallmatrix}\mathcal{L}_\lambda & 0 \\ 0 &
\mathcal{R}_\lambda \end{psmallmatrix}$.  Since $f_\lambda>0$,
$\mathcal{L}_\lambda\succ 0$, and $\mathcal{R}_\lambda\succ 0$ by
hypothesis, we have $\mathsf{Q}_N\succ\mathsf{0}$ for all sufficiently
large $N$ according to Proposition \ref{detQ}. This shows that
\begin{equation*}
1-\lambda\frac{\langle z,\mathsf{M}_Nz\rangle}{\langle z,\Sigma_N^{-1}z\rangle}>0
\end{equation*}
for all sufficiently large $N$ and $z\in(\Cm^d)^{N+2}$. Thus, by
taking the infimum over $z$ we have
\begin{equation*}
1-\lambda\sup_{\substack{z\in(\Cm^d)^{N+2}\\ z\neq 0}}\bigg\{\frac{\langle z,\mathsf{M}_Nz\rangle}{\langle z,\Sigma_N^{-1}z\rangle}\bigg\}>0
\end{equation*}
if $\lambda\ge 0$ and
\begin{equation*}
1-\lambda\inf_{\substack{z\in(\Cm^d)^{N+2}\\ z\neq 0}}\bigg\{\frac{\langle z,\mathsf{M}_Nz\rangle}{\langle z,\Sigma_N^{-1}z\rangle}\bigg\}>0
\end{equation*}
if $\lambda<0$. By sending $N$ to infinity we realize that
$\lambda\xi_+\le 1$ if $\lambda\ge 0$, which also gives
$\lambda\xi_-\le 1$ as $\xi_-\le\xi_+$, and that $\lambda\xi_-\le 1$
if $\lambda<0$, which also gives $\lambda\xi_+\le 1$.

Conversely, if $\lambda\in\Rl$ is such that $\lambda\xi_-<1$ and
$\lambda\xi_+<1$, then there exists $\epsilon>0$ such that
$1-\lambda\xi_-\ge2\epsilon$ and $1-\lambda\xi_+\ge2\epsilon$. This
yields that for all sufficiently large $N$ and $z\in(\Cm^d)^{N+2}$
\begin{equation*}
\langle z,\Sigma_N^{-1}z\rangle-\lambda\langle z,\mathsf{M}_Nz\rangle\ge \epsilon\langle z,\Sigma_N^{-1}z\rangle.
\end{equation*}
This way, $r(\mathsf{Q}_N)\ge \epsilon r(\Sigma_N^{-1})\ge
\epsilon\sigma$ for all sufficiently large $N$, where
$\mathsf{Q}_N:=\Sigma_N^{-1}-\lambda\mathsf{M}_N$ and Lemma
\ref{lem:norm} has been invoked. It follows from Lemma \ref{primo}
that $r(\mathsf{T}_N)\ge\epsilon \sigma$ and
$r(\mathsf{S}_N)\ge\epsilon \sigma$ for all sufficiently large $N$, so
that
$\inf_{\theta\in[0,2\pi]}\{r(F_\lambda(\theta))\}\ge\epsilon\sigma$ by
Lemma \ref{secondo} and $r(S_\infty)\ge\epsilon\sigma$ by Lemma
\ref{lem:terzo} with
$\mathsf{S}_\infty=\begin{psmallmatrix}\mathcal{L}_\lambda & 0 \\ 0 &
\mathcal{R}_\lambda \end{psmallmatrix}$. Thus, $f_\lambda>0$,
$\mathcal{L}_\lambda\succ 0$, and $\mathcal{R}_\lambda\succ 0$, so
that $\lambda\in\Lambda$.
\end{proof}

The limit (\ref{limCGF}) together with the fact that $\Lambda$ is an
interval finally give the following important result. We stress that
$\Lambda$ contains an open neighborhood of the origin, as it is
manifest by Lemma \ref{lem:lambdaset}, so that
$\lambda_-<0<\lambda_+$.
\begin{proposition}
  \label{prop:CGF}
  For all $\lambda\in(\lambda_-,\lambda_+)$
  \begin{equation*}
    \lim_{N\uparrow\infty}\frac{1}{N}\ln\Ex\big[e^{\lambda W_N}\big]=-\frac{1}{4\pi}\int_0^{2\pi}\ln\det F_\lambda(\theta) \,d\theta.
  \end{equation*}
\end{proposition}

The function $\varphi$ that maps any $\lambda\in(\lambda_-,\lambda_+)$
in
\begin{equation*}
    \varphi(\lambda):=-\frac{1}{4\pi}\int_0^{2\pi}\ln\det F_\lambda(\theta) \,d\theta
  \end{equation*}
is convex. A rapid way to demonstrate this fact is to observe that
$\varphi$ is the limit of a sequence of convex functions by
Proposition \ref{prop:CGF}.  The function $\varphi$ is also
differentiable since $F_\lambda(\theta)$ is differentiable with
respect to $\lambda$ for each $\theta\in[0,2\pi]$. The asymptotic
theory of sequences of convex functions (see \cite{Rockafellar},
Theorem 24.5) gives for all $\lambda\in(\lambda_-,\lambda_+)$
\begin{equation}
\lim_{N\uparrow\infty}\frac{1}{N}\frac{d}{d\lambda}\ln\Ex\big[e^{\lambda W_N}\big]=\varphi'(\lambda).
\label{eq:limder}
\end{equation}
Limit (\ref{eq:limder}) will serve us to verify the lower large
deviation bound.  Other notable consequences of convexity are that the
limits $\lim_{\lambda\downarrow\lambda_-}\varphi(\lambda)=:\varphi_-$
and $\lim_{\lambda\uparrow\lambda_+}\varphi(\lambda)=:\varphi_+$ exist
(see \cite{Rockafellar}, Theorem 7.5) and that $\varphi'$ is
non-decreasing, in such a way that also the limits
$\lim_{\lambda\downarrow\lambda_-}\varphi'(\lambda)=:d_-$ and
$\lim_{\lambda\uparrow\lambda_+}\varphi'(\lambda)=:d_+$ exist.

\begin{remark}
\label{rem}
In Section \ref{sec:mainres} we have claimed that
$\lim_{N\uparrow\infty}(1/N)\ln\Ex[e^{\lambda W_N}]=+\infty$ if
$\lambda\notin\overline{(\lambda_-,\lambda_+)}=\bar{\Lambda}$.
Although we do not need this limit to prove a large deviation
principle, we can verify it as follows. Assume for instance that
$\lambda_+<+\infty$ and pick $\lambda>\lambda_+>0$. It must be
$\lambda\xi_+>1$ since, on the contrary,
$\lambda\xi_-\le\lambda\xi_+\le 1$ and $\lambda\in\bar{\Lambda}$ as a
consequence according to Lemma \ref{lem:lambdaset}. Fix $N\ge 1$ and
let $\mathsf{A}\in\mathsf{BL}_{N+2,N+2}$ be a real invertible matrix
such that $\Sigma_N=\mathsf{A}\mathsf{A}^\top$, which exists because
$\Sigma_N\succ 0$. Denoting by $m_1,\ldots,m_{(N+2)d}$ the eigenvalues
of the real symmetric matrix
$\mathsf{A}^\top\mathsf{M}_N\mathsf{A}\in\mathsf{BL}_{N+2,N+2}$, we
have
\begin{align}
  \nonumber
  \xi_N:=\max\big\{m_1,\ldots,m_{(N+2)d}\big\}
  &=\sup_{\substack{z\in(\Cm^d)^{N+2}\\z\ne 0}}\bigg\{\frac{\langle z,\mathsf{A}^\top\mathsf{M}_N\mathsf{A}z\rangle}{\langle z,z\rangle}\bigg\}
  =\sup_{\substack{z\in(\Cm^d)^{N+2}\\z\ne 0}}\bigg\{\frac{\langle \mathsf{A}z,\mathsf{M}_N\mathsf{A}z\rangle}{\langle z,z\rangle}\bigg\}\\
  \nonumber
  &=\sup_{\substack{z\in(\Cm^d)^{N+2}\\z\ne 0}}\bigg\{\frac{\langle z,\mathsf{M}_Nz\rangle}{\langle \mathsf{A}^{-1}z,\mathsf{A}^{-1}z\rangle}\bigg\}
  =\sup_{\substack{z\in(\Cm^d)^{N+2}\\z\ne 0}}\bigg\{\frac{\langle z,\mathsf{M}_Nz\rangle}{\langle z,\Sigma_N^{-1}z\rangle}\bigg\}.
\end{align}
Thus, Lemma \ref{lem:lambdaset} tells us that the number
$\lambda\xi_N$ approaches $\lambda\xi_+>1$ at large $N$, so that
$1-\lambda\xi_N\le 0$ if $N$ exceeds a threshold value $N_o$. This
shows that the matrix
$\Sigma_N^{-1}-\lambda\mathsf{M}_N=(\mathsf{A}^{-1})^\top(\mathsf{I}-\lambda\mathsf{A}^\top\mathsf{M}_N\mathsf{A})\mathsf{A}^{-1}$
is not positive-definite for $N>N_o$. Lemma \ref{lem:Gaussint}
concludes the proof.
\end{remark}

\subsection{The upper large deviation bound}
\label{sec:upper}

In this section we prove the upper large deviation bound for closed
sets. We start with some standard results from the theory of large
deviations that we shall use to prove both the upper large deviation
bound and the lower large deviation bound.  For each
$\eta\in(\lambda_-,\lambda_+)$ and $N\ge 1$, let $\prob_{\eta,N}$ be
the probability measure on $(\Omega,\mathscr{F})$ defined by the
exponential change of measure
\begin{equation}
\frac{d\prob_{\eta,N}}{d\prob}:=\frac{e^{\eta W_N}}{\Ex[e^{\eta W_N}]}.
\label{new_prob}
\end{equation}
Let $\varphi_\eta$ be the function that maps any $\lambda\in\Rl$ in
\begin{equation}
  \varphi_\eta(\lambda):=\limsup_{N\uparrow\infty}\frac{1}{N}\ln\Ex_{\eta,N}\big[e^{\lambda W_N}\big].
\label{new_phi}
\end{equation}
Since $0\in(\lambda_-,\lambda_+)$, we have $\prob_{0,N}=\prob$ for all
$N\ge 1$ and $\varphi_0(\lambda)=\varphi(\lambda)$ for all
$\lambda\in(\lambda_-,\lambda_+)$, $\varphi$ being the convex
differentiable function introduced at the end of the last
section. Moreover, if $\lambda\in\Rl$ is such that
$\lambda+\eta\in(\lambda_-,\lambda_+)$, then
\begin{equation*}
  \varphi_\eta(\lambda)=\limsup_{N\uparrow\infty}\frac{1}{N}\ln\frac{\Ex[e^{(\lambda+\eta) W_N}]}{\Ex[e^{\eta W_N}]}
  =\varphi(\lambda+\eta)-\varphi(\eta)<+\infty.
\end{equation*}
It follows that the function $\varphi_\eta$ is finite and
differentiable in an open neighborhood of the origin with
$\varphi_\eta'(0)=\varphi'(\eta)$. The following lemma states an upper
large deviation bound with respect to the measure $\prob_{\eta,N}$. We
recall that the Fenchel-Legendre transform $I_\eta$ of $\varphi_\eta$
is the convex function that associates $w\in\Rl$ with
\begin{equation*}
I_\eta(w):=\sup_{\lambda\in\Rl}\big\{w\lambda-\varphi_\eta(\lambda)\big\}.
\end{equation*}
\begin{lemma}
  \label{GET}
  Fix $\eta\in(\lambda_-,\lambda_+)$. The following conclusions hold:
  \begin{enumerate}
  \item the Fenchel-Legendre transform $I_\eta$ of $\varphi_\eta$ has compact level sets;
  \item for each closed set $\mathcal{F}\subseteq\Rl$
  \begin{equation*}
    \limsup_{N\uparrow\infty}\frac{1}{N}\ln\prob_{\eta,N}\bigg[\frac{W_N}{N}\in \mathcal{F}\bigg]\le-\inf_{w\in\mathcal{F}}\{I_\eta(w)\};
  \end{equation*}
  \item for each $\epsilon>0$ there exists $\kappa>0$ such that for
    all sufficiently large $N$
    \begin{equation*}
    \prob_{\eta,N}\bigg[\bigg|\frac{W_N}{N}-\varphi'(\eta)\bigg|\ge\epsilon\bigg]\le e^{-\kappa N}.
  \end{equation*}
  \end{enumerate}
\end{lemma}

\begin{proof}
$I_\eta$ is lower semicontinuous as any Fenchel-Legendre transform
  (see \cite{Rockafellar}, Theorem 12.2). Due to lower semicontinuity,
  the level sets of $I_\eta$ are closed. In order to prove part 1, it
  remains to verify that they are bounded. As the function
  $\varphi_\eta$ is finite in an open neighborhood of the origin,
  there exists $\delta>0$ such that $\varphi_\eta(\delta)<+\infty$ and
  $\varphi_\eta(-\delta)<+\infty$. If $I_\eta(w)\le a$ for given real
  numbers $w$ and $a$, then $w\delta-\varphi_\eta(\delta)\le
  I_\eta(w)\le a$ and $-w\delta-\varphi_\eta(-\delta)\le I_\eta(w)\le
  a$ by definition, that is $-[a+\varphi_\eta(-\delta)]/\delta\le w\le
  [a+\varphi_\eta(\delta)]/\delta$.
  
  Part 2 is a standard result from large deviation theory (see
  \cite{Dembo}, Theorem 2.3.6 and Exercise 2.3.25). In a nutshell, the
  upper large deviation bound for compact sets is a manipulation of
  the Chernoff bound and holds without any assumption on the function
  $\varphi_\eta$. Extension to all closed sets is made possible by
  finiteness of $\varphi_\eta$ in an open neighborhood of the origin,
  which entails exponential tightness.

  As far as part 3 is concerned, in the light of part 2 it suffices to
  demonstrate that
  \begin{equation}
    \inf_{v\notin(w-\epsilon,w+\epsilon)}\{I_\eta(v)\}>0
\label{part3_Upper}
  \end{equation}
  for each $\epsilon>0$ with $w:=\varphi'(\eta)$. To begin with, let
  us observe that $I_\eta(v)>0$ if $v\ne w$. On the contrary, if
  $I_\eta(v)=0$, then for all $\lambda$ in a neighborhood of the
  origin we would have
  $\varphi(\lambda+\eta)-\varphi(\eta)=\varphi_\eta(\lambda)\ge
  v\lambda$ by definition of $I_\eta(v)$. This would imply
  $v=\varphi'(\eta)=:w$, which contradicts the hypothesis $v\ne w$. We
  can now verify (\ref{part3_Upper}). Pick $\epsilon>0$ and notice
  that the set $\mathcal{A}:=\{v\in\Rl:I_\eta(v)\le 1\}$ is compact by
  part 1. If $(w-\epsilon,w+\epsilon)^c\,\cap\mathcal{A}=\emptyset$,
  then $\inf_{v\notin(w-\epsilon,w+\epsilon)}\{I_\eta(v)\}\ge 1$. If
  $(w-\epsilon,w+\epsilon)^c\,\cap\mathcal{A}\ne\emptyset$, then there
  exists $v_\star\in(w-\epsilon,w+\epsilon)^c\,\cap\mathcal{A}$ such
  that $I_\eta(v)\ge I_\eta(v_\star)$ for all
  $v\in(w-\epsilon,w+\epsilon)^c\,\cap\mathcal{A}$, and hence for all
  $v\notin(w-\epsilon,w+\epsilon)$, as $I_\eta$ is a lower
  semicontinuous function and
  $(w-\epsilon,w+\epsilon)^c\,\cap\mathcal{A}$ is a compact set. On
  the other hand, we have $I_\eta(v_\star)>0$ since $v_\star\ne w$.
\end{proof}

Lemma \ref{GET} gives the following upper large deviation bound for
the quadratic functionals $W_N$ of the stable Gauss-Markov processes
$\X$.
\begin{proposition}
  The following conclusions hold:
  \begin{enumerate}
  \item the convex function $I$ that maps $w\in\Rl$ in $I(w):=\sup_{\lambda\in(\lambda_-,\lambda_+)}\{w\lambda-\varphi(\lambda)\}$ has compact level sets;
  \item for each closed set $\mathcal{F}\subseteq\Rl$
  \begin{equation*}
    \limsup_{N\uparrow\infty}\frac{1}{N}\ln\prob\bigg[\frac{W_N}{N}\in \mathcal{F}\bigg]\le-\inf_{w\in\mathcal{F}}\{I(w)\}.
  \end{equation*}
  \end{enumerate}
\end{proposition}

\begin{proof}
  $I$ is the Fenchel-Legendre transform of the function that
  associates $\lambda\in(\lambda_-,\lambda_+)$ with $\varphi(\lambda)$
  and $\lambda\notin(\lambda_-,\lambda_+)$ with $+\infty$. Then, part
  1 is proved in the same way of part 1 of Lemma \ref{GET}.  Part 2
  follows from part 2 of Lemma \ref{GET} with $\eta=0$ as
  $I_0(w):=\sup_{\lambda\in\Rl}\{w\lambda-\varphi_0(\lambda)\}\ge\sup_{\lambda\in(\lambda_-,\lambda_+)}\{w\lambda-\varphi(\lambda)\}
  =:I(w)$ for every $w\in\Rl$.
\end{proof}

\subsection{The lower large deviation bound}
\label{sec:lower}

In this section we prove the lower large deviation bound for open
sets, namely that for each open set $\mathcal{G}\subseteq\Rl$
  \begin{equation*}
    \liminf_{N\uparrow\infty}\frac{1}{N}\ln\prob\bigg[\frac{W_N}{N}\in \mathcal{G}\bigg]\ge-\inf_{w\in\mathcal{G}}\{I(w)\},
  \end{equation*}
  where $I$ is the function that maps $w\in\Rl$ in
  $I(w):=\sup_{\lambda\in(\lambda_-,\lambda_+)}\{w\lambda-\varphi(\lambda)\}$.
  This is tantamount to state that for all $w\in\Rl$ and $\delta>0$
\begin{equation}
\liminf_{N\uparrow\infty}\frac{1}{N}\ln \prob\bigg[\frac{W_N}{N}\in (w-\delta,w+\delta)\bigg]\ge -I(w).
\label{lower_bound}
\end{equation}
We start with the following lower bound based on Lemma \ref{GET}.
\begin{lemma}
  \label{lower_diff}
  Fix $w\in\Rl$ and assume that there exists $\eta\in(\lambda_-,\lambda_+)$
  such that $w=\varphi'(\eta)$. Then, for every $\delta>0$
  \begin{equation*}
\liminf_{N\uparrow\infty}\frac{1}{N}\ln \prob\bigg[\frac{W_N}{N}\in (w-\delta,w+\delta)\bigg]\ge \varphi(\eta)-w\eta.
\end{equation*}
\end{lemma}

\begin{proof}
Let $\prob_{\eta,N}$ and $\varphi_\eta$ be the probability measure
(\ref{new_prob}) and the function (\ref{new_phi}), respectively.  Fix
$\delta>0$ and pick $\epsilon\in(0,\delta)$.  The fact that $\eta
W_N-Nw\eta-N\epsilon |\eta|\le 0$ if $W_N/N\in(w-\epsilon,w+\epsilon)$
gives for each $N\ge 1$
\begin{align}
  \nonumber
  \prob\bigg[\frac{W_N}{N}\in(w-\delta,w+\delta)\bigg]&\ge
  e^{-Nw\eta -N\epsilon |\eta|}\,\Ex\bigg[e^{\eta W_N}\mathds{1}_{\big\{\frac{W_N}{N}\in(w-\epsilon,w+\epsilon)\big\}}\bigg]\\
  \nonumber
  &=e^{-Nw\eta -N\epsilon |\eta|}\,\Ex\big[e^{\eta W_N}\big]\,\prob_{\eta,N}\bigg[\bigg|\frac{W_N}{N}-w\bigg|<\epsilon\bigg],
\end{align}
and part 3 of Lemma \ref{GET} shows that
 \begin{equation*}
    \lim_{N\uparrow\infty}\prob_{\eta,N}\bigg[\bigg|\frac{W_N}{N}-w\bigg|<\epsilon\bigg]=1.
 \end{equation*}
Thus, by invoking Proposition \ref{prop:CGF} we obtain
\begin{equation*}
  \liminf_{N\uparrow\infty}\frac{1}{N}\ln\prob\bigg[\frac{W_N}{N}\in(w-\delta,w+\delta)\bigg]\ge \varphi(\eta)-w\eta -\epsilon|\eta|.
\end{equation*}
The lemma follows from here by sending $\epsilon$ to 0.
\end{proof}

Lemma \ref{lower_diff} allows us to demonstrate the lower large
deviation bound (\ref{lower_bound}) for $w$ in the closure
$\overline{(d_-,d_+)}$ of $(d_-,d_+)$, where
$d_-:=\lim_{\lambda\downarrow\lambda_-}\varphi'(\lambda)$ and
$d_+:=\lim_{\lambda\uparrow\lambda_+}\varphi'(\lambda)$ as in Section
\ref{sec:cum}. Notice that convexity and differentiability of
$\varphi$ yield
$\varphi(\lambda)\ge\varphi(\eta)+\varphi'(\eta)(\lambda-\eta)$ for
every $\lambda$ and $\eta$ in $(\lambda_-,\lambda_+)$, so that
$I(w)=w\eta-\varphi(\eta)$ if $w=\varphi'(\eta)$ for some
$\eta\in(\lambda_-,\lambda_+)$.  Since $d_-\le\varphi'(0)\le d_+$ as
$\varphi'$ is non-decreasing, if $d_-=d_+$, then
$\overline{(d_-,d_+)}$ contains only $\varphi'(0)$ and bound
(\ref{lower_bound}) directly follows from Lemma \ref{lower_diff} with
$\eta=0$. If $d_-<d_+$ and $w\in(d_-,d_+)$, then there exists
$\eta\in(\lambda_-,\lambda_+)$ such that $w=\varphi'(\eta)$ and bound
(\ref{lower_bound}) follows again from Lemma \ref{lower_diff} with
such $\eta$. If $d_-<d_+<+\infty$, then $\overline{(d_-,d_+)}$
contains $d_+$, and we tackle the case $w=d_+$ as follows. Fix
$\delta>0$. There exist $v\in(d_-,d_+)$ arbitrarily close to $w$ and
$\epsilon>0$ such that
$(v-\epsilon,v+\epsilon)\subseteq(w-\delta,w+\delta)$. This way, since
(\ref{lower_bound}) holds for $v$ we find
\begin{equation*}
  \liminf_{N\uparrow\infty}\frac{1}{N}\ln \prob\bigg[\frac{W_N}{N}\in (w-\delta,w+\delta)\bigg]\ge
  \liminf_{N\uparrow\infty}\frac{1}{N}\ln \prob\bigg[\frac{W_N}{N}\in (v-\epsilon,v+\epsilon)\bigg]\ge-I(v). 
\end{equation*}
From here, we get bound (\ref{lower_bound}) for $w=d_+$ by sending $v$
to $w$ and by observing that $\lim_{v\uparrow w}I(v)=I(w)$ by
convexity and lower semicontinuity of $I$ (see \cite{Rockafellar},
Corollary 7.5.1). Similar arguments can be used to solve the case
$-\infty<d_-<d_+$ and $w=d_-$.

In order to complete the proof of the lower large deviation bound
(\ref{lower_bound}), it remains to address the case $d_+<+\infty$ and
$w>d_+$, as well as the case $d_->-\infty$ and $w<d_-$. They are
similar, so that we discuss in detail the former only, omitting the
proof of the latter. Assume that $d_+<+\infty$ and fix $w>d_+$. We
claim that the case $\lambda_+=+\infty$ is trivial, so that we also
suppose $\lambda_+<+\infty$. In fact, convexity and differentiability
of $\varphi$ combined with $\varphi(0)=0$ give
$\varphi(\lambda)\le\lambda\varphi'(\lambda)$ for all
$\lambda\in(\lambda_-,\lambda_+)$. It follows that $I(w)\ge
w\lambda-\varphi(\lambda)\ge\lambda[w-\varphi'(\lambda)]$ for all
$\lambda\in(\lambda_-,\lambda_+)$. Thus, if $\lambda_+=+\infty$, then
we realize that $I(w)=+\infty$ by sending $\lambda$ to $\lambda_+$, as
$\lim_{\lambda\uparrow\lambda_+}\varphi'(\lambda)=d_+<w$, and the
lower bound (\ref{lower_bound}) is trivial. Observe that if
$\lambda_+<+\infty$ and $d_+<+\infty$, then
$\lim_{\lambda\uparrow\lambda_+}\varphi(\lambda)=:\varphi_+<+\infty$
as $\varphi(\lambda)\le\lambda\varphi'(\lambda)$ for all
$\lambda\in(\lambda_-,\lambda_+)$. Since the function that associates
$\lambda\in(\lambda_-,\lambda_+)$ with $w\lambda-\varphi(\lambda)$ is
increasing under the hypothesis $w>d_+$, we have
$I(w):=\sup_{\lambda\in(\lambda_-,\lambda_+)}\{w\lambda-\varphi(\lambda)\}=w\lambda_+-\varphi_+$.

The idea to prove (\ref{lower_bound}) for $w>d_+$ and
$\lambda_+<+\infty$ is to make a change of measure like in Lemma
\ref{lower_diff}, but this time the parameter $\eta$ must depend on
the time $N$. Let us introduce such parameter. Pick $N\ge 1$.  Since
the covariance matrix $\Sigma_N\in\mathsf{BL}_{N+2,N+2}$ is symmetric
positive-definite, there exists a real invertible matrix
$\mathsf{A}\in\mathsf{BL}_{N+2,N+2}$ such that
$\Sigma_N=\mathsf{A}\mathsf{A}^\top$. Like in Remark \ref{rem}, let
$m_1,\ldots,m_{(N+2)d}$ be the eigenvalues of the real symmetric
matrix $\mathsf{A}^\top\mathsf{M}_N\mathsf{A}\in\mathsf{BL}_{N+2,N+2}$
and observe that
\begin{align}
  \nonumber
  \xi_N:=\max\big\{m_1,\ldots,m_{(N+2)d}\big\}
  &=\sup_{\substack{z\in(\Cm^d)^{N+2}\\z\ne 0}}\bigg\{\frac{\langle z,\mathsf{A}^\top\mathsf{M}_N\mathsf{A}z\rangle}{\langle z,z\rangle}\bigg\}\\
  \nonumber
  &=\sup_{\substack{z\in(\Cm^d)^{N+2}\\z\ne 0}}\bigg\{\frac{\langle z,\mathsf{M}_Nz\rangle}{\langle z,(\mathsf{A}^{-1})^\top\mathsf{A}^{-1}z\rangle}\bigg\}
  =\sup_{\substack{z\in(\Cm^d)^{N+2}\\z\ne 0}}\bigg\{\frac{\langle z,\mathsf{M}_Nz\rangle}{\langle z,\Sigma_N^{-1}z\rangle}\bigg\}.
\end{align}
Similarly
\begin{equation*}
  \min\big\{m_1,\ldots,m_{(N+2)d}\big\}
  =\inf_{\substack{z\in(\Cm^d)^{N+2}\\z\ne 0}}\bigg\{\frac{\langle z,\mathsf{M}_Nz\rangle}{\langle z,\Sigma_N^{-1}z\rangle}\bigg\}.
\end{equation*}
We have $\lim_{N\uparrow\infty}\xi_N=\xi_+$ by Lemma
\ref{lem:lambdaset}, whereas $\min\{m_1,\ldots,m_{(N+2)d}\}$
approaches $\xi_-$ at large $N$. Lemma \ref{lem:lambdaset} also gives
$\xi_+>0$ since $\lambda_+<+\infty$ by hypothesis. Indeed, $\xi_+\le
0$ would entail that the set $\Lambda$ contains all positive real
numbers.  Lemma \ref{lem:Gaussint} and the fact that
$\det\Sigma_N=\det\Sigma_o$ show that if
$\mathsf{I}-\lambda\mathsf{A}^\top\mathsf{M}_N\mathsf{A}\succ 0$,
namely if $1-\lambda m_l>0$ for $l=1,\ldots,(N+2)d$, then
\begin{equation}
  \ln\Ex\big[e^{\lambda W_{N+2}}\big]=-\frac{1}{2}\ln\det\big(\mathsf{I}-\lambda\mathsf{A}^\top\mathsf{M}_N\mathsf{A}\big)
  =-\frac{1}{2}\sum_{l=1}^{(N+2)d}\ln(1-\lambda m_l).
\label{eq:CGFeig}
\end{equation}
We claim that for all sufficiently large $N$ there exists
$\eta_N\in(0,\xi_N)$ such that
\begin{equation}
\frac{1}{2(N+2)}\sum_{l=1}^{(N+2)d}\frac{m_l}{1-\eta_N m_l}=w.
\label{eqetaN}
\end{equation}
Notice that $\xi_N>0$ for all sufficiently large $N$ as
$\lim_{N\uparrow\infty}\xi_N=\xi_+>0$.  In fact, identity
(\ref{eq:CGFeig}) in combination with (\ref{eq:limder}) yields that
$[2(N+2)]^{-1}\sum_{l=1}^{(N+2)d}m_l$ approaches $\varphi'(0)\le
d_+<w$ when $N$ is sent to infinity.  Thus, for all sufficiently large
$N$, the continuous function that maps $\lambda\in[0,\xi_N)$ in
  $[2(N+2)]^{-1}\sum_{l=1}^{(N+2)d}m_l(1-\lambda m_l)^{-1}$ increases
  from a value smaller than $w$ at $\lambda=0$ to $+\infty$ at
  $\lambda=\xi_N$, so that there exists a unique $\eta_N$ satisfying
  (\ref{eqetaN}).  We must have
  $\lim_{N\uparrow\infty}\eta_N=\lambda_+$. On the contrary, there
  would exist $\epsilon>0$ and a diverging sequence $\{N_k\}_{k\ge 0}$
  of positive integers such that $\eta_{N_k}<\lambda_+-\epsilon$ for
  all $k\ge 0$. Then, for every $k$
\begin{equation*}
  w=\frac{1}{2(N_k+2)}\sum_{l=1}^{(N_k+2)d}\frac{m_l}{1-\eta_{N_k} m_l}\le\frac{1}{2(N_k+2)}\sum_{l=1}^{(N_k+2)d}\frac{m_l}{1-(\lambda_+-\epsilon) m_l}.
\end{equation*}
By sending $k$ to infinity and by combining (\ref{eq:CGFeig}) with
(\ref{eq:limder}), from here we would get
$w\le\varphi'(\lambda_+-\epsilon)\le d_+$, which contradicts the
assumption $w>d_+$. Another property of $\eta_N$ is that
\begin{equation}
\liminf_{N\uparrow\infty}\frac{1}{N+2}\ln\Ex\big[e^{\eta_N W_{N+2}}\big]\ge \varphi_+.
\label{eq:liminfeta}
\end{equation}
In order to verify this bound, fix $\lambda\in(0,\lambda_+)$ and bear
in mind that $\eta_N\ge\lambda$ for all sufficiently large $N$ as
$\lim_{N\uparrow\infty}\eta_N=\lambda_+$, so that
\begin{equation*}
  -\ln(1-\eta_N m_l)\ge-\ln(1-\lambda m_l)+(\eta_N-\lambda)\min\{0,m_1,\ldots,m_{(N+2)d}\}
\end{equation*}
for $l$ and sufficiently large $N$. Then, for all sufficiently large
$N$ we have
\begin{align}
  \nonumber
  \frac{1}{N+2}\ln\Ex\big[e^{\eta_N W_{N+2}}\big]&=-\frac{1}{2(N+2)}\sum_{l=1}^{(N+2)d}\ln(1-\eta_N m_l)\\
  \nonumber
  &\ge-\frac{1}{2(N+2)}\sum_{l=1}^{(N+2)d}\ln(1-\lambda m_l)+(\eta_N-\lambda)d\min\{0,m_1,\ldots,m_{(N+2)d}\}\\
  \nonumber
  &=\frac{1}{N+2}\ln\Ex\big[e^{\lambda W_{N+2}}\big]+(\eta_N-\lambda)d\min\{0,m_1,\ldots,m_{(N+2)d}\}.
\end{align}
By sending $N$ to infinity and by recalling that
$\min\{m_1,\ldots,m_{(N+2)d}\}$ approaches $\xi_-$ in this limit,
Proposition \ref{prop:CGF} shows that
\begin{equation*}
\liminf_{N\uparrow\infty}\frac{1}{N+2}\ln\Ex\big[e^{\eta_N W_{N+2}}\big]\ge \varphi(\lambda)+(\lambda_+-\lambda)d\min\{0,\xi_-\},
\end{equation*}
which demonstrates (\ref{eq:liminfeta}) once $\lambda$ is sent
$\lambda_+$.

We now move to bound (\ref{lower_bound}) and put $\eta_N$ into
context. Fix $\delta>0$ and pick $\epsilon\in(0,\delta)$. For all
sufficiently large $N$, $\eta_N$ is positive as
$\lim_{N\uparrow\infty}\eta_N=\lambda_+$ and we have
\begin{align}
  \nonumber
  \prob\bigg[\frac{W_{N+2}}{N+2}\in(w-\delta,w+\delta)\bigg]&\ge
  e^{-(N+2)(w+\epsilon)\eta_N}\,\Ex\bigg[e^{\eta_N W_{N+2}}\mathds{1}_{\big\{\frac{W_{N+2}}{N+2}\in(w-\epsilon,w+\epsilon)\big\}}\bigg]\\
  \nonumber
  &=e^{-(N+2)(w+\epsilon)\eta_N}\,\Ex\big[e^{\eta_N W_{N+2}}\big]\,\prob_{\eta_N,N+2}\bigg[\bigg|\frac{W_{N+2}}{N+2}-w\bigg|<\epsilon\bigg],
\end{align}
where $\prob_{\eta_N,N+2}$ is the probability measure (\ref{new_prob})
associated with $\eta_N$.  This bound, together with
(\ref{eq:liminfeta}), yields
\begin{align}
  \nonumber
  \liminf_{N\uparrow\infty}\frac{1}{N}\ln\prob\bigg[\frac{W_N}{N}\in(w-\delta,w+\delta)\bigg]&\ge \varphi_+-w\lambda_+-\epsilon\lambda_+\\
  \nonumber
  &+\liminf_{N\uparrow\infty}\frac{1}{N+2}\ln\prob_{\eta_N,N+2}\bigg[\bigg|\frac{W_{N+2}}{N+2}-w\bigg|<\epsilon\bigg].
\end{align}
This way, as $w\lambda_+-\varphi_+=I(w)$, we get at the lower large
deviation bound (\ref{lower_bound}) from here if we can prove that
\begin{equation}
 \lim_{\epsilon\downarrow 0}\,\liminf_{N\uparrow\infty}\frac{1}{N+2}\ln\prob_{\eta_N,N+2}\bigg[\bigg|\frac{W_{N+2}}{N+2}-w\bigg|<\epsilon\bigg]=0.
\label{eq:limetaNultimo}
\end{equation}
Verifying (\ref{eq:limetaNultimo}) is our last task. To this aim, we
resort to the following result, which was introduced by Bryc and Dembo
(see \cite{Dembo_Sol}, Lemma 2) to deal with a similar problem.
\begin{lemma}
\label{lem:Dembo}
If $\{Z_l\}_{l\ge 1}$ is a sequence of i.i.d.\ random variables with
mean zero, finite second moment, and positive probability density
function at 0 with respect to a probability measure $P$, then for each
$\epsilon>0$ there exists $p>0$ such that the following property
holds:
\begin{equation*}
P\Bigg[\Bigg|\sum_{l\ge 1}a_lZ_l\Bigg|<\epsilon\Bigg]\ge p
\end{equation*}
for any numerical sequence $\{a_l\}_{l\ge 1}$ such that $\sum_{l\ge 1}|a_l|\le 1$.
\end{lemma}

Let $\{Y_l\}_{l\ge 1}$ be a sequence of independent standard Gaussian
random variables with respect to a probability measure $P$.  Lemma
\ref{lem:Dembo} ensures that for each $\epsilon>0$ there exists $p>0$
with the property that
\begin{equation}
P\Bigg[\Bigg|\sum_{l\ge 1}a_l\big(Y_l^2-1\big)\Bigg|<\frac{\epsilon}{1+|w|+3d|\xi_-|}\Bigg]\ge p
\label{bound_Bryc}
\end{equation}
for any numerical sequence $\{a_l\}_{l\ge 1}$ such that $\sum_{l\ge
  1}|a_l|\le 1$, $\xi_-$ being the number introduced by Lemma
\ref{lem:lambdaset}. We make use of property (\ref{bound_Bryc}) to
prove (\ref{eq:limetaNultimo}). Since the real symmetric matrix
$\mathsf{I}-\eta_N\mathsf{A}^\top\mathsf{M}_N\mathsf{A}\in\mathsf{BL}_{N+2,N+2}$
has positive eigenvalues $1-\eta_Nm_1,\ldots,1-\eta_Nm_{(N+2)d}$, if
we build a diagonal matrix $\mathsf{D}$ with
$\sqrt{1-\eta_Nm_1},\ldots,\sqrt{1-\eta_Nm_{(N+2)d}}$ on the diagonal,
then
$\mathsf{I}-\eta_N\mathsf{A}^\top\mathsf{M}_N\mathsf{A}=\mathsf{O}^\top\mathsf{D}^2\mathsf{O}$
with an orthogonal matrix $\mathsf{O}\in\mathsf{BL}_{N+2,N+2}$. This
way, if we write $W_{N+2}=(1/2)\langle X,\mathsf{M}_NX\rangle$ with
$X:=(X_1,\ldots,X_{N+2})$, then standard manipulations of Gaussian
integrals yield for all $k\in(\Rl^d)^{N+2}$
\begin{align}
  \nonumber
  \Ex_{\eta_N,N+2}\big[e^{\mathrm{i}\langle k,\mathsf{D}\mathsf{O}\mathsf{A}^{-1}X\rangle}\big]&=
  \frac{\Ex[e^{\mathrm{i}\langle (\mathsf{A}^\top)^{-1}\mathsf{O}^\top\mathsf{D}k,X\rangle+\frac{1}{2}\eta_N\langle X,\mathsf{M}_NX\rangle}]}{\Ex[e^{\frac{1}{2}\eta_N\langle X,\mathsf{M}_NX\rangle}]}\\
  \nonumber
  &=e^{-\frac{1}{2}\langle (\mathsf{A}^\top)^{-1}\mathsf{O}^\top\mathsf{D}k,(\Sigma_N^{-1}-\eta_N\mathsf{M}_N)^{-1}(\mathsf{A}^\top)^{-1}\mathsf{O}^\top\mathsf{D}k\rangle}\\
  \nonumber
  &=e^{-\frac{1}{2}\langle \mathsf{O}^\top\mathsf{D}k,(\mathsf{I}-\eta_N\mathsf{A}^\top\mathsf{M}_N\mathsf{A})^{-1}\mathsf{O}^\top\mathsf{D}k\rangle}=e^{-\frac{1}{2}\langle k,k\rangle}.
\end{align}
This formula states that the characteristic function of the random
vector $Y:=\mathsf{D}\mathsf{O}\mathsf{A}^{-1}X$ with respect to the
probability measure $\prob_{\eta_N,N+2}$ is the characteristic
function of $(N+2)d$ independent standard Gaussian random variables.
Thus, the components $Y_1,\ldots,Y_{(N+2)d}$ of $Y$ are independent
standard Gaussian random variables with respect to the probability
measure $\prob_{\eta_N,N+2}$. It follows from (\ref{bound_Bryc}) that
for each $\epsilon>0$ there exists $p>0$ with the property that
\begin{equation}
\prob_{\eta_N,N+2}\Bigg[\Bigg|\sum_{l=1}^{(N+2)d}a_l\big(Y_l^2-1\big)\Bigg|<\frac{\epsilon}{1+|w|+3d|\xi_-|}\Bigg]\ge p
\label{bound_Bryc1}
\end{equation}
for all $N\ge 1$ and real numbers $a_1,\ldots,a_{(N+2)d}$ such that
$\sum_{l=1}^{(N+2)d}|a_l|\le 1$.  Let us observe now that
\begin{equation*}
  W_{N+2}=\frac{1}{2}\langle  X,\mathsf{M}_NX\rangle
  =\frac{1}{2}\langle\mathsf{O}^\top\mathsf{D}^{-1}Y,\mathsf{A}^\top\mathsf{M}_N\mathsf{A}\mathsf{O}^\top\mathsf{D}^{-1}Y\rangle
  =\frac{1}{2}\sum_{l=1}^{(N+2)d}\frac{m_l}{1-\eta_N m_l}Y_l^2.
\end{equation*}
This identity combined with (\ref{eqetaN}) shows that for all
$\epsilon>0$ and sufficiently large $N$
\begin{equation}
  \prob_{\eta_N,N+2}\bigg[\bigg|\frac{W_{N+2}}{N+2}-w\bigg|<\epsilon\bigg]=
  \prob_{\eta_N,N+2}\Bigg[\Bigg|\sum_{l=1}^{(N+2)d}a_l\big(Y_l^2-1\big)\Bigg|<\frac{\epsilon}{1+|w|+3d|\xi_-|}\Bigg],
\label{bound_Bryc2}
\end{equation}
where, for $l=1,\ldots,(N+2)d$, we have set
\begin{equation*}
a_l:=\frac{1}{2(N+2)(1+|w|+3d|\xi_-|)}\frac{m_l}{1-\eta_N m_l}.
\end{equation*}
We have $\sum_{l=1}^{(N+2)d}|a_l|\le 1$ for all sufficiently large
$N$. In fact, since
\begin{equation*}
\frac{|m_l|}{1-\eta_N m_l}\le\frac{m_l}{1-\eta_N m_l}-2\min\big\{0,m_1,\ldots,m_{(N+2)d}\big\}
\end{equation*}
for every $l$, by invoking (\ref{eqetaN}) and by recalling that
$\min\{m_1,\ldots,m_{(N+2)d}\}$ approaches $\xi_-$ at large $N$, for
all sufficiently large $N$ we find
\begin{align}
  \nonumber
  (1+|w|+3d|\xi_-|)\sum_{l=1}^{(N+2)d}|a_l|&=\frac{1}{2(N+2)}\sum_{l=1}^{(N+2)d}\frac{|m_l|}{1-\eta_N m_l}\\
  \nonumber
  &\le\frac{1}{2(N+2)}\sum_{l=1}^{(N+2)d}\frac{m_l}{1-\eta_N m_l}-2d\min\big\{0,m_1,\ldots,m_{(N+2)d}\big\}\\
  \nonumber
  &\le w+2d\Big|\min\big\{m_1,\ldots,m_{(N+2)d}\big\}\Big|\le  |w|+3d|\xi_-|.
\end{align}
In conclusion, by comparing (\ref{bound_Bryc2}) with
(\ref{bound_Bryc1}) we realize that for each $\epsilon>0$ there exists
$p>0$ such that 
\begin{equation*}
  \prob_{\eta_N,N+2}\bigg[\bigg|\frac{W_{N+2}}{N+2}-w\bigg|<\epsilon\bigg]\ge p
\end{equation*}
for all sufficiently large $N$. This bound proves
(\ref{eq:limetaNultimo}).

\section{Proof of Theorem \ref{LDP_eN}}
\label{sec:simmetria}

We know that the entropy production $Ne_N$ is the quadratic functional
$W_N$ corresponding to the matrices $L:=I-\Sigma_o^{-1}-S^\top S$,
$U:=0$, $R:=\Sigma_o^{-1}+S^\top S-I$, and $V:=S-S^\top$. Part 1 of
the theorem then follows from Theorem \ref{main} with the rate
function $I$ that maps $w\in\Rl$ in
$\sup_{\lambda\in(\lambda_-,\lambda_+)}\{w\lambda-\varphi(\lambda)\}$.
It remains to verify the Gallavotti-Cohen symmetry stated by part 2.

Formula (\ref{F_e}) shows that the Hermitian matrices
$F_\lambda(\theta)$ associated with the entropy production satisfy
$F_{-\lambda-1}(\theta)=F_\lambda(2\pi-\theta)$ for all
$\lambda\in\Rl$ and $\theta\in[0,2\pi]$. According to (\ref{def:infr})
and (\ref{def:phi}), this identity immediately gives
$f_{-\lambda-1}=f_\lambda$ for any $\lambda$ and
$\varphi(-\lambda-1)=\varphi(\lambda)$ for any $\lambda$ such that
$f_{-\lambda-1}=f_\lambda>0$. We shall show in a moment that
$\lambda_-=-\lambda_+-1$. It follows that for every $w\in\Rl$
\begin{align}
  \nonumber
  I(-w)-w&=\sup_{\lambda\in(\lambda_-,\lambda_+)}\big\{w(-\lambda-1)-\varphi(\lambda)\big\}\\
  \nonumber
  &=\sup_{\lambda\in(-\lambda_+-1,-\lambda_--1)}\big\{w\lambda-\varphi(-\lambda-1)\big\}=\sup_{\lambda\in(\lambda_-,\lambda_+)}\big\{w\lambda-\varphi(\lambda)\big\}=I(w),
\end{align}
which demonstrates the Gallavotti-Cohen symmetry for $I$.

Let us verify that $\lambda_-=-\lambda_+-1$. Recalling that
$\lambda_-=\inf\{\Lambda\}$ and $\lambda_+=\sup\{\Lambda\}$, $\Lambda$
being the set defined by (\ref{def:Lambda_set}), it suffices to prove
that $-\lambda-1\in\Lambda$ whenever $\lambda\in\Lambda$. Fix
$\lambda\in\Lambda$. Then, $f_{-\lambda-1}=f_\lambda>0$, which implies
that the matrices $\Phi_{-\lambda-1}(n)$ given by (\ref{def:Fourierc})
and $H_{-\lambda-1}$, $K_{-\lambda-1}$, $\mathcal{L}_{-\lambda-1}$,
and $\mathcal{R}_{-\lambda-1}$ introduced by Lemma \ref{lem:tech} are
well-defined. The identity
$F_{-\lambda-1}(\theta)=F_\lambda(2\pi-\theta)$ shows that
$\Phi_{-\lambda-1}(n)=\Phi_\lambda(-n)=\Phi_\lambda^\dagger(n)$ for
all $n\in\mathbb{Z}$. The latter entails that $H_{-\lambda-1}$ and
$K_\lambda$ are related by the law
\begin{equation*}
 H_{-\lambda-1}=I+\big[(\lambda+1)S^\top-\lambda S\big]\Phi_{-\lambda-1}(1)=I+\big[(\lambda+1)S^\top-\lambda S\big]\Phi_\lambda^\dagger(1)=K_\lambda^\dagger.
\end{equation*}
This law induces a relationship between the matrices
$\mathcal{L}_{-\lambda-1}$ and $\mathcal{R}_\lambda$. In fact
\begin{align}
  \nonumber
  \mathcal{L}_{-\lambda-1}&=(\lambda+1)I-\lambda(\Sigma_o^{-1}+S^\top S)
  -\big[(\lambda+1)S-\lambda S^\top\big]\Phi_{-\lambda-1}(0)H_{-\lambda-1}^{-1}\big[(\lambda+1)S^\top-\lambda S\big]\\
  \nonumber
  &=(\lambda+1)I-\lambda(\Sigma_o^{-1}+S^\top S)
  -\big[(\lambda+1)S-\lambda S^\top\big]\Phi_\lambda(0)(K_\lambda^{-1})^\dagger\big[(\lambda+1)S^\top-\lambda S\big],
\end{align}
which, by taking adjoint on both the sides and by bearing in mind that
$\mathcal{L}_\lambda$ is Hermitian, yields
\begin{equation*}
\mathcal{L}_{-\lambda-1}=(\lambda+1)I-\lambda(\Sigma_o^{-1}+S^\top S)
  -\big[(\lambda+1)S-\lambda S^\top\big]K_\lambda^{-1}\Phi_\lambda(0)\big[(\lambda+1)S^\top-\lambda S\big]=\mathcal{R}_\lambda.
\end{equation*}
Since $\mathcal{R}_\lambda\succ 0$ by hypothesis, we obtain
$\mathcal{L}_{-\lambda-1}\succ 0$. By similar arguments, we find that
$K_{\lambda-1}=H_\lambda^\dagger$ and
$\mathcal{R}_{-\lambda-1}=\mathcal{L}_\lambda\succ 0$. In conclusion,
$f_{-\lambda-1}>0$, $\mathcal{L}_{-\lambda-1}\succ 0$, and
$\mathcal{R}_{-\lambda-1}\succ 0$, so that $-\lambda-1\in\Lambda$.

\appendix

\section{Proof of Lemma \ref{lem:quad}}
\label{proof:quad}

Fix $N\ge 1$. Let $\mu_N^+:=\prob[(X_1,\ldots,X_N)\in\cdot\,]$ and
$\mu_N^-:=\prob[(X_N,\ldots,X_1)\in\cdot\,]$ be the probability
measures on the Borel sets of $(\Rl^d)^N$ induced by the Gauss-Markov
chain $\X$. According to (\ref{evol_eq}), $\mu_N^+$ and $\mu_N^-$ are
the multivariate Gaussian distributions that have densities
\begin{equation*}
  \frac{d\mu_N^+}{d\ell}(x_1,\ldots,x_N):=
  \frac{e^{-\frac{1}{2}\langle x_1, \Sigma_o^{-1} x_1\rangle}}{\sqrt{(2\pi)^d\det\Sigma_o}}\prod_{n=2}^N\frac{1}{\sqrt{(2\pi)^d}}\,e^{-\frac{1}{2}\|x_n-Sx_{n-1}\|^2}
\end{equation*}
and
\begin{equation*}
  \frac{d\mu_N^-}{d\ell}(x_1,\ldots,x_N):=\frac{d\mu_N^+}{d\ell}(x_N,\ldots,x_1)=
  \frac{e^{-\frac{1}{2}\langle x_N, \Sigma_o^{-1} x_N\rangle}}{\sqrt{(2\pi)^d\det\Sigma_o}}\prod_{n=2}^N\frac{1}{\sqrt{(2\pi)^d}}\,e^{-\frac{1}{2}\|x_{n-1}-Sx_n\|^2}
\end{equation*}
with respect to the Lebesgue measure $\ell$. Thus, $\mu_N^+\ll
\mu_N^-\ll\ell$ and standard results about measure theory \cite{Rudin}
give for all $(x_1,\ldots,x_N)\in(\Rl^d)^N$
\begin{align}
  \nonumber
\ln\bigg[\frac{d\mu_N^+}{d\mu_N^-}(x_1,\ldots,x_N)\bigg]&=\ln\bigg[\frac{d\mu_N^+}{d\ell}(x_1,\ldots,x_N)\bigg/\frac{d\mu_N^-}{d\ell}(x_1,\ldots,x_N)\bigg]\\
\nonumber
&=\frac{1}{2}\langle x_1,(I-\Sigma_o^{-1}-S^\top S)x_1\rangle+\frac{1}{2}\langle x_N,(\Sigma_o^{-1}+S^\top S-I)x_N\rangle \\
    \nonumber
    &+\sum_{n=2}^N\langle x_n,(S-S^\top)x_{n-1}\rangle.
\end{align}

\section{Proof of Lemma \ref{lem:norm}}
\label{proof:norm}

For all $N\ge 1$ and $z=(z_1,\ldots,z_{N+2})\in(\Cm^d)^{N+2}$ we have
\begin{align}
\nonumber
\big\|\mathsf{Q}_Nz\big\|^2&=\big\|Az_1+E^\dagger z_2\big\|^2+\sum_{n=2}^{N+1}\big\|Ez_{n-1}+Dz_n+E^\dagger z_{n+1}\big\|^2+\big\|Ez_{N+1}+Bz_{N+2}\big\|^2\\
\nonumber
&\le\Big(\|A\|\|z_1\|+\|E\|\|z_2\|\Big)^2+\sum_{n=2}^{N+1}\Big(\|E\|\|z_{n-1}\|+\|D\|\|z_n\|+\|E\|\|z_{n+1}\|\Big)^2\\
\nonumber
&+\Big(\|E\|\|z_{N+1}\|+\|B\|\|z_{N+2}\|\Big)^2\\
\nonumber
&\le 2\|A\|^2\|z_1\|^2+2\|E\|^2\|z_2\|^2+3\sum_{n=2}^{N+1}\Big(\|E\|^2\|z_{n-1}\|^2+\|D\|^2\|z_n\|^2+\|E\|^2\|z_{n+1}\|^2\Big)\\
\nonumber
&+2\|E\|^2\|z_{N+1}\|^2+2\|B\|^2\|z_{N+2}\|^2\\
\nonumber
&\le\Big(2\|A\|^2+3\|D\|^2+2\|B\|^2+6\|E\|^2\Big)\|z\|^2.
\end{align}

\section{Proof of Lemma \ref{primo}}
\label{proof:primo}

Fix $N\ge 1$.  Assume that there exists $q>0$ such that $\langle
z,\mathsf{Q}_Nz\rangle\ge q\langle z,z\rangle$ for all
$z\in(\Cm^d)^{N+2}$.  Bearing in mind (\ref{QconT}) and by writing $z$
as $(a,t_1,\ldots,t_N,b)$ with $s:=(a,b)\in(\Cm^d)^2$ and
$t:=(t_1,\ldots,t_N)\in(\Cm^d)^N$, this condition reads
\begin{equation*}
  \langle a,Aa\rangle+2\langle t,\mathsf{C}Ea\rangle+\langle t,\mathsf{T}_Nt\rangle
  +2\langle b,E\mathsf{R}t\rangle+\langle b,Bb\rangle\ge q\langle a,a\rangle+q\langle t,t\rangle+q\langle b,b\rangle.
\end{equation*}
This way, by setting $a:=0$ and $b:=0$ we find $\langle
t,\mathsf{T}_Nt\rangle\ge q\langle t,t\rangle$ for any
$t\in(\Cm^d)^N$. This shows in particular that $\mathsf{T}_N$ is
invertible. By setting
$t:=-\mathsf{T}_N^{-1}\mathsf{C}Ea-\mathsf{T}_N^{-1}\mathsf{R}^\dagger
E^\dagger b$ we obtain
\begin{equation*}
  \langle a,\big(A-E^\dagger\mathsf{C}^\dagger\mathsf{T}_N^{-1}\mathsf{C}E\big)a\rangle
  -2\langle a,E^\dagger\mathsf{C}^\dagger\mathsf{T}_N^{-1}\mathsf{R}^\dagger E^\dagger b\rangle
  +\langle b,\big(B-E\mathsf{R}\mathsf{T}_N^{-1}\mathsf{R}^\dagger E^\dagger \big)b\rangle\ge q\langle a,a\rangle+q\langle b,b\rangle,
\end{equation*}
that is $\langle s,\mathsf{S}_Ns\rangle\ge q\langle s,s\rangle$ for
all $s\in(\Cm^d)^2$.  Part 1 is thus verified.

As far as part 2 is concerned, if $\mathsf{T}_N$ is invertible,
then we can write down the identity
\begin{equation}
\mathsf{Q}_N=
\mathsf{L}^\dagger
\begin{pmatrix}
   A-E^\dagger\mathsf{C}^\dagger\mathsf{T}_N^{-1}\mathsf{C}E & \mathsf{0}& -E^\dagger\mathsf{C}^\dagger\mathsf{T}_N^{-1}\mathsf{R}^\dagger E^\dagger \\
     \mathsf{0} & \mathsf{T}_N & \mathsf{0} \\
    -E\mathsf{R}\mathsf{T}_N^{-1}\mathsf{C}E & \mathsf{0}& B-E\mathsf{R}\mathsf{T}_N^{-1}\mathsf{R}^\dagger E^\dagger
\end{pmatrix}
\mathsf{L}
\label{QconL}
\end{equation}
with
\begin{equation*}
\mathsf{L}:=
\begin{pmatrix}
   I & \mathsf{0} & 0 \\
    \mathsf{T}_N^{-1}\mathsf{C}E & \mathsf{I} & \mathsf{T}_N^{-1}\mathsf{R}^\dagger E^\dagger \\
    0 & \mathsf{0} & I
\end{pmatrix}
\in\mathsf{BL}_{N+2,N+2}.
\end{equation*}
Since $\det \mathsf{L}=1$, it follows by permutations of rows and
columns that
\begin{align}
  \nonumber
\det \mathsf{Q}_N&=
\det\begin{pmatrix}
   A-E^\dagger\mathsf{C}^\dagger\mathsf{T}_N^{-1}\mathsf{C}E & \mathsf{0}& -E^\dagger\mathsf{C}^\dagger\mathsf{T}_N^{-1}\mathsf{R}^\dagger E^\dagger \\
     \mathsf{0} & \mathsf{T}_N & \mathsf{0} \\
    -E\mathsf{R}\mathsf{T}_N^{-1}\mathsf{C}E & \mathsf{0}& B-E\mathsf{R}\mathsf{T}_N^{-1}\mathsf{R}^\dagger E^\dagger
\end{pmatrix}\\
\nonumber
&=\det\begin{pmatrix}
  \mathsf{T}_N & \mathsf{0} & \mathsf{0} \\
  \mathsf{0} & A-E^\dagger\mathsf{C}^\dagger\mathsf{T}_N^{-1}\mathsf{C}E & -E^\dagger\mathsf{C}^\dagger\mathsf{T}_N^{-1}\mathsf{R}^\dagger E^\dagger \\
  \mathsf{0} & -E\mathsf{R}\mathsf{T}_N^{-1}\mathsf{C}E & B-E\mathsf{R}\mathsf{T}_N^{-1}\mathsf{R}^\dagger E^\dagger
\end{pmatrix}
=\det \mathsf{T}_N\cdot\det\mathsf{S}_N.
\end{align}
Moreover, given $z\in(\Cm^d)^{N+2}$, by writing
$Lz=(a,t_1,\ldots,t_N,b)$ with $s:=(a,b)\in(\Cm^d)^2$ and
$t:=(t_1,\ldots,t_N)\in(\Cm^d)^N$, we realize from (\ref{QconL}) that
$\langle z,\mathsf{Q}_Nz\rangle=\langle t,\mathsf{T}_Nt\rangle+\langle
s,\mathsf{S}_Ns\rangle$. Thus, if $\mathsf{T}_N\succ\mathsf{0}$,
$\mathsf{S}_N\succ\mathsf{0}$, and $z\ne0$, then we have $\langle
z,\mathsf{Q}_Nz\rangle>0$ since $\mathsf{L}$ is invertible.

\section{Proof of Lemma \ref{secondo}}
\label{proof:secondo}

Part 1 is immediate since $r(\mathsf{T}_N)$ is non-increasing with
respect to $N$. In fact, given any $z=(z_1,\ldots,z_N)\in(\Cm^d)^N$,
by setting $\zeta:=(z_1,\ldots,z_N,0)\in(\Cm^d)^{N+1}$ we see that
$r(\mathsf{T}_{N+1})\langle z,z\rangle=r(\mathsf{T}_{N+1})\langle
\zeta,\zeta\rangle\le \langle \zeta,\mathsf{T}_{N+1}\zeta\rangle=
\langle z,\mathsf{T}_Nz\rangle$.  Part 3 is nothing but the Szeg\"o
theorem for the determinant of Hermitian block Toeplitz matrices (see
\cite{Toeplitz}, Theorem 7). Let us focus on part 2. Assume that
$r(\mathsf{T}_N)\ge t$ for all $N\ge 1$ and pick a positive continuous
function $\varphi$ with period $2\pi$ and a vector $u\in\Cm^d$. Due to
the assumed properties of $\varphi$, there exists a sequence
$\{p_N\}_{N\ge 0}$ of trigonometric polynomials that converges
uniformly to $\sqrt{\varphi}$, $p_N$ having degree $N$ (see
\cite{Rudin}, Theorem 4.25).  Write $p_N(\theta)$ as $\sum_{n=-N}^N
c_{N,n} e^{-\mathrm{i}n\theta}$ for each $N$ and $\theta$. Since
$r(\mathsf{T}_{2N+1})\ge t$, by setting
$\zeta_n=z_n:=c_{N,N-n+1}e^{\mathrm{i}(N+1)}u$ for
$n=1,\ldots,2N+1$ in (\ref{TFrel}) we obtain
\begin{align}
  \nonumber
  \frac{1}{2\pi}\int_0^{2\pi} \big\langle u,F(\theta)u\big\rangle\, p_N^2(\theta)\, d\theta&=
  \frac{1}{2\pi}\int_0^{2\pi} \Bigg\langle\sum_{n=1}^{2N+1}z_ne^{-\mathrm{i}n\theta},
  F(\theta)\sum_{n=1}^{2N+1} z_ne^{-\mathrm{i}n\theta}\Bigg\rangle \,d\theta\\
  \nonumber
  &=\langle z,\mathsf{T}_{2N+1}z\rangle\ge t\sum_{n=1}^{2N+1} \langle z_n,z_n\rangle=t\langle u,u\rangle\,\frac{1}{2\pi}\int_0^{2\pi}  p_N^2(\theta)\, d\theta
\end{align}
for all $N\ge 0$. By sending $N$ to infinity we get
\begin{equation*}
  \frac{1}{2\pi}\int_0^{2\pi} \big\langle u,F(\theta)u\big\rangle\, \varphi(\theta)\, d\theta
  \ge t\langle u,u\rangle\,\frac{1}{2\pi}\int_0^{2\pi} \varphi(\theta)\, d\theta.
\end{equation*}
The arbitrariness of $\varphi$ and $u$ shows that $\langle
z,F(\theta)z\rangle\ge t\langle z,z\rangle$ for all
$\theta\in[0,2\pi]$ and $z\in\Cm^d$.

Conversely, if $r(F(\theta))\ge t$ for every $\theta\in[0,2\pi]$, then
by invoking (\ref{TFrel}) again we can write for all $N\ge 1$ and
$z=(z_1,\ldots,z_N)\in(\Cm^d)^N$
  \begin{equation*}
\langle z,\mathsf{T}_Nz\rangle
  \ge \frac{t}{2\pi}\int_0^{2\pi} \Bigg\langle\sum_{n=1}^Nz_ne^{-\mathrm{i}n\theta},
  \sum_{n=1}^N z_ne^{-\mathrm{i}n\theta} \Bigg\rangle \,d\theta=t\sum_{n=1}^N \langle z_n,z_n\rangle.
\qedhere
  \end{equation*}

\section{Proof of Lemma \ref{lem:terzo}}
\label{proof:terzo}

Suppose for a moment that the matrix $H$ is invertible. Then, the
matrix $K$ is proved to be invertible by contradiction. In fact, if
$K$ is not invertible, then there exists a vector $u\in\Cm^d$
different from $0$ such that $Ku=[I-\Phi(1)E]u=0$.  We must have
$Eu\ne 0$, otherwise $u=0$. Since $HE=EK$, we get $HEu=0$ with $Eu\ne
0$, which contradicts the assumption that $H$ is invertible.

Let us demonstrate now that the matrix $H$ is invertible. This will
prove part 1 of the lemma. We proceed by contradiction. Suppose that
there exists a vector $u\in\Cm^d$ different from $0$ such that $Hu=0$.
Pick an arbitrary integer $N\ge 3$ and for $n=1,\ldots,N$ consider the
vectors
\begin{equation*}
  z_n:=\Phi(1-n)u=\frac{1}{2\pi}\int_0^{2\pi}F^{-1}(\theta)e^{\mathrm{i}(n-1)\theta}d\theta \,u.
\end{equation*}
We have $z_1\ne 0$ since $\Phi(0)$ is invertible.  We claim that
\begin{equation}
\mathsf{T}_N\begin{pmatrix}
z_1 \\
z_2\\
   \vdots \\
   z_{N-1} \\
   z_N
\end{pmatrix}
=
\begin{pmatrix}
Dz_1+E^\dagger z_2 \\
 E z_1+Dz_2+E^\dagger z_3\\
   \vdots \\
 E z_{N-2}+Dz_{N-1}+E^\dagger z_N \\
   Ez_{N-1}+Dz_N
\end{pmatrix}
=\begin{pmatrix}
0\\
   0 \\
   \vdots \\
   0 \\
   -E^\dagger \Phi(-N)u
\end{pmatrix}.
\label{tech1}
\end{equation}
Indeed, for $n=2,\ldots,N-1$ we have
\begin{align}
  \nonumber
  E z_{n-1}+Dz_n+E^\dagger z_{n+1}&
  =\frac{1}{2\pi}\int_0^{2\pi}\big[Ee^{-\mathrm{i}\theta}+D+E^\dagger e^{\mathrm{i}\theta}\big]F^{-1}(\theta)e^{\mathrm{i}(n-1)\theta} d\theta \,u\\
\nonumber
  &=\frac{1}{2\pi}\int_0^{2\pi}e^{\mathrm{i}(n-1)\theta} d\theta \,u=0
\end{align}
and
\begin{align}
  \nonumber
  Dz_1+E^\dagger z_2&=\frac{1}{2\pi}\int_0^{2\pi}\big[D+E^\dagger e^{\mathrm{i}\theta}\big]F^{-1}(\theta)\,d\theta \,u\\
  \nonumber
  &=\frac{1}{2\pi}\int_0^{2\pi}\big[F(\theta)-E e^{-\mathrm{i}\theta}\big]F^{-1}(\theta)\,d\theta \,u=\big[I-E\Phi(1)\big]u=Hu=0.
\end{align}
Finally, we see that
\begin{align}
  \nonumber
  Ez_{N-1}+Dz_N&=\frac{1}{2\pi}\int_0^{2\pi}\big[Ee^{-\mathrm{i}\theta}+D\big]F^{-1}(\theta)e^{\mathrm{i}(N-1)\theta} d\theta \,u\\
  \nonumber
  &=\frac{1}{2\pi}\int_0^{2\pi}\big[F(\theta)-E^\dagger e^{\mathrm{i}\theta}\big]F^{-1}(\theta)e^{\mathrm{i}(N-1)\theta} d\theta \,u
  =-E^\dagger \Phi(-N)u.
\end{align}
Due to (\ref{tech1}), it follows from (\ref{TFrel}) with
$\zeta=z:=(z_1,\ldots,z_N)\in(\Cm^d)^N$ and the hypothesis
$t:=\inf_{\theta\in[0,2\pi]}\{r(F(\theta))\}>0$ that
\begin{equation*}
  -\Big\langle E^\dagger \Phi(1-N)u, E^\dagger \Phi(-N)u\Big\rangle =\langle z,\mathsf{T}_Nz\rangle
  \ge t\sum_{n=1}^N \langle z_n,z_n\rangle\ge t\langle z_1,z_1\rangle.
\end{equation*}
This bound is absurd since $z_1\ne 0$ and
$\lim_{N\uparrow\infty}\Phi(-N)=0$ by the Riemann-Lebesgue lemma.

Let us move to part 2.  As
$\inf_{\theta\in[0,2\pi]}\{r(F(\theta))\}>0$, $\mathsf{T}_N$ is
invertible by Lemma \ref{secondo} and we can set
\begin{equation}
  \mathsf{T}_N^{-1}\mathsf{C}=:
  \begin{pmatrix}
    C_1 \\
    \vdots \\
    C_N \\
  \end{pmatrix}
  \label{ITC}
\end{equation}
and
\begin{equation}
  \mathsf{R}\mathsf{T}_N^{-1}=:
  \begin{pmatrix}
    R_1 & \cdots & R_N 
  \end{pmatrix}.
  \label{ITR}
\end{equation}
The matrices $\mathsf{C}$ and $\mathsf{R}$ were defined in
(\ref{defCmat}) and (\ref{defRmat}), respectively.  We have
$\mathsf{C}^\dagger \mathsf{T}_N^{-1}\mathsf{C}=C_1$ and
$\mathsf{R}\mathsf{T}_N^{-1}\mathsf{R}^\dagger=R_N$, which on the one
hand show that $C_1$ and $R_N$ are Hermitian and on the other hand
allow us to write
\begin{equation*}
\mathsf{S}_N=\begin{pmatrix}
    A-E^\dagger C_1E &  -E^\dagger R_1^\dagger E^\dagger \\
    -EC_NE &  B-ER_NE^\dagger
\end{pmatrix}.
\end{equation*}
Let us verify that $C_1$ approaches the matrix $\Phi(0)H^{-1}$ and
$R_N$ approaches the matrix $K^{-1}\Phi(0)$ when $N$ is sent to
infinity, whereas $C_N$ and $R_1$ approach 0. These facts prove part 2
of the lemma.

To begin with, we observe that since $H^\dagger$ and $K$ are
non-singular and $\lim_{N\uparrow\infty}\Phi(\pm N)=0$ by the
Riemann-Lebesgue lemma, the matrix
\begin{equation}
  \mathsf{Z}:=\begin{pmatrix}
    H^\dagger  & -\Phi(N)E \\
    -\Phi(-N)E^\dagger & K
  \end{pmatrix}
  \in\mathsf{BL}_{2,2}
  \label{defGmat}
\end{equation}
is invertible if $N>N_o$, $N_o\ge 2$ being a sufficiently large
integer.  Pick $N>N_o$. By multiplying (\ref{ITC}) by $\mathsf{T}_N$
on the left and (\ref{ITR}) by $\mathsf{T}_N$ on the right we
explicitly have
\begin{equation}
  \begin{cases}
    DC_1+E^\dagger C_2=I & \mbox{for }n=1,\\
    EC_{n-1}+DC_n+E^\dagger C_{n+1}=0 & \mbox{for }n=2,\ldots,N-1,\\
    EC_{N-1}+DC_N=0 & \mbox{for }n=N
  \end{cases}
  \label{eqTC}
\end{equation}
and
\begin{equation}
  \begin{cases}
   DR_1^\dagger+E^\dagger R_2^\dagger=0 & \mbox{for }n=1,\\
   ER_{n-1}^\dagger+DR_n^\dagger+E^\dagger R_{n+1}^\dagger=0 & \mbox{for }n=2,\ldots,N-1,\\
   ER_{N-1}^\dagger+DR_N^\dagger=I & \mbox{for }n=N.
  \end{cases}
  \label{eqTR}
\end{equation}
By multiplying the $n$th equation in (\ref{eqTC}) by
$e^{-\mathrm{i}n\theta}$ and then by carrying out the sum over $n$ we
get
\begin{equation*}
 F(\theta)\sum_{n=1}^N C_ne^{-\mathrm{i}n\theta}=Ie^{-\mathrm{i}\theta}+E^\dagger C_1+EC_Ne^{-\mathrm{i}(N+1)\theta},
\end{equation*}
which gives for $n=1,\ldots,N$
\begin{align}
  \nonumber
  C_n&=\frac{1}{2\pi}\int_0^{2\pi}F^{-1}(\theta)\Big[Ie^{\mathrm{i}(n-1)\theta}+E^\dagger C_1e^{\mathrm{i}n\theta}+EC_Ne^{\mathrm{i}(n-N-1)\theta}\Big]d\theta\\
  &=\Phi(1-n)+\Phi(-n)E^\dagger C_1+\Phi(N-n+1)EC_N.
  \label{eqTC1}
\end{align}
Similarly, (\ref{eqTR}) shows that for $n=1,\ldots,N$
\begin{equation}
  R_n^\dagger=\Phi(N-n)+\Phi(-n)E^\dagger R_1^\dagger+\Phi(N-n+1)ER_N^\dagger.
  \label{eqTR1}
\end{equation}
At this point, by setting $n:=1$ and $n:=N$ in (\ref{eqTC1}) and by
recalling that $C_1=C_1^\dagger$ we realize that
\begin{equation*}
  \mathsf{Z}
  \begin{pmatrix}
    C_1^\dagger \\
    C_N
  \end{pmatrix}
  =\begin{pmatrix}
    \Phi(0) \\
    \Phi(1-N)
  \end{pmatrix},
\end{equation*}
$\mathsf{Z}$ being the matrix defined in (\ref{defGmat}).
It follows that
\begin{equation*}
  \begin{pmatrix}
    C_1^\dagger \\
    C_N
  \end{pmatrix}
  =\mathsf{Z}^{-1}
  \begin{pmatrix}
    \Phi(0) \\
    \phi(1-N)
  \end{pmatrix}
\end{equation*}
as $\mathsf{Z}$ is invertible for $N>N_o$. Similarly, (\ref{eqTR1})
for $n:=1$ and $n:=N$ and the fact that $R_N^\dagger=R_N$ yield
\begin{equation*}
  \begin{pmatrix}
    R_1^\dagger \\
    R_N
  \end{pmatrix}
  =\mathsf{Z}^{-1}\begin{pmatrix}
    \Phi(N-1) \\
    \Phi(0)
  \end{pmatrix}.
\end{equation*}
This way, the Riemann-Lebesgue lemma entails that $C_1$ approaches
$\Phi(0)H^{-1}$ and $R_N$ approaches $K^{-1}\Phi(0)$ when $N$ is sent
to infinity, whereas $C_N$ and $R_1$ approach 0.

\section{Proof of Lemma \ref{lem:sigma}}
\label{proof:sigma}

As the spectral radius $\rho(S)$ of $S$ is smaller than $1$ by
hypothesis, Gelfand's formula for spectral radii gives
$\lim_{n\uparrow\infty}\|S^n\|^{\frac{1}{n}}=\rho(S)<1$. Then, there
exist $s\in(0,1)$ and a positive constant $c$ such that $\|S^n\|\le c
s^n$ for all $n\ge 0$. Let us show that the lemma holds with
$\sigma:=[1\wedge r(\Sigma_o^{-1})](1-s)^2c^{-2}$, which is positive
since, obviously, $r(\Sigma_o^{-1})>0$.  Fix $N\ge 1$ and, to begin
with, observe that
\begin{align}
  \nonumber
  r(\Sigma_N^{-1})&=\inf_{\substack{z\in(\Cm^d)^{N+2}\\z\ne 0}}\bigg\{\frac{\langle z_1,\Sigma_o^{-1}z_1\rangle+\sum_{n=2}^{N+2}\|z_n-Sz_{n-1}\|^2}
  {\sum_{n=1}^{N+2}\langle z_n,z_n\rangle}\bigg\}\\
  \nonumber
  &\ge 1\wedge r(\Sigma_o^{-1})\inf_{\substack{z\in(\Cm^d)^{N+2}\\z\ne 0}}\bigg\{\frac{\|z_1\|^2+\sum_{n=2}^{N+2}\|z_n-Sz_{n-1}\|^2}{\sum_{n=1}^{N+2}\langle z_n,z_n\rangle}\bigg\}.
\end{align}
Since for any $z=(z_1, \ldots, z_{N+2})\in(\Cm^d)^{N+2}$ there exists
$\zeta=(\zeta_1, \ldots, \zeta_{N+2})\in(\Cm^d)^{N+2}$ such that
$z_n=\sum_{k=1}^n S^{n-k}\zeta_k$ for each $n$, this bound yields
\begin{equation*}
  r(\Sigma_N^{-1})\ge 1\wedge r(\Sigma_o^{-1})\inf_{\substack{\zeta\in(\Cm^d)^{N+2}\\\zeta\ne 0}}\bigg\{\frac{\sum_{n=1}^{N+2}\|\zeta_n\|^2}
      {\sum_{n=1}^{N+2}\sum_{h=1}^n\sum_{k=1}^n\langle S^{n-h}\zeta_h,S^{n-k}\zeta_k\rangle}\bigg\}.
\end{equation*}
At this point, it suffices to invoke the Cauchy-Schwarz inequality to
conclude that for every $\zeta=(\zeta_1, \ldots,
\zeta_{N+2})\in(\Cm^d)^{N+2}$
\begin{align}
  \nonumber
  \sum_{n=1}^{N+2}\sum_{h=1}^n\sum_{k=1}^n\langle S^{n-h}\zeta_h,S^{n-k}\zeta_k\rangle &\le \sum_{n=1}^{N+2}\sum_{h=1}^n\sum_{k=1}^n\|S^{n-h}\zeta_h\|\|S^{n-k}\zeta_k\|\\
  \nonumber
  &\le c^2\sum_{n=1}^{N+2}\sum_{h=1}^n\sum_{k=1}^n s^{2n-h-k}\|\zeta_h\|\|\zeta_k\|\\
  \nonumber
  &\le \frac{c^2}{2}\sum_{n=1}^{N+2}\sum_{h=1}^n\sum_{k=1}^n s^{2n-h-k}\Big(\|\zeta_h\|^2+\|\zeta_k\|^2\Big)\\
  \nonumber
  &\le\frac{c^2}{(1-s)^2}\sum_{k=1}^{N+2} \|\zeta_k\|^2.
\qedhere
\end{align}

\section*{Acknowledgements}
The authors are grateful to Giuseppe Gonnella for suggesting the
problem of large deviations for the entropy production rate and
quadratic functionals of autoregressive models.

\section*{Data availability}

Data sharing is not applicable to this article as no new data were
created or analyzed in this study.

\end{document}